\documentclass[11pt]{amsart}
\usepackage{amsmath,amsfonts,amssymb,amsthm,epsfig}
\usepackage{ifpdf}
\usepackage[all]{xy}
\usepackage{tikz}

\newcommand{\uuv[1]}{
\begin{picture}(1.5,0.7)
\put(0,-0.4){
\begin{tikzpicture}[scale=0.7]
\draw[line width = 0.06cm] (0,0.3) -- (0,0.8);
\draw[line width = 0.06cm] (1,0.3) -- (1,0.8);
\put(0.25,0.33){$#1$};
\draw[line width = 0.06cm] (0,0.55) -- (0.2,0.55);
\draw[line width = 0.06cm] (0.8,0.55) -- (1,0.55);
\draw[->, line width = 0.06cm] (0.5,0.8) -- (0.5,1.1);
\draw[->, line width = 0.06cm, white] (0.5,0.3) -- (0.5,0);
\end{tikzpicture}}
\end{picture}
}

\newcommand{\udv[1]}{
\begin{picture}(1.5,0.7)
\put(0,-0.4){
\begin{tikzpicture}[scale=0.7]
\draw[line width = 0.06cm] (0,0.3) -- (0,0.8);
\draw[line width = 0.06cm] (1,0.3) -- (1,0.8);
\put(0.25,0.33){$#1$};
\draw[line width = 0.06cm] (0,0.55) -- (0.25,0.55);
\draw[line width = 0.06cm] (0.75,0.55) -- (1,0.55);
\draw[->, line width = 0.06cm, white] (0.5,0.8) -- (0.5,1.1);
\draw[->, line width = 0.06cm] (0.5,0.3) -- (0.5,0);
\end{tikzpicture}}
\end{picture}
}

\newcommand{\suv[1]}{
\begin{picture}(1.5,0.7)
\put(0,-0.4){
\begin{tikzpicture}[scale=0.7]
\fill[black!35!white] (0,0.3) .. controls (0.5,0.3) and (0.5,0.3).. (1,0.3) ..  controls (1,0.55) and (1,0.55) .. (1,0.8)
 ..  controls (0.5,0.8) and (0.5,0.8) .. (0,0.8) ;
\draw[line width = 0.06cm] (0,0.3) -- (0,0.8);
\draw[line width = 0.06cm] (1,0.3) -- (1,0.8);
\put(0.25,0.33){$#1$};
\draw[line width = 0.06cm] (0,0.55) -- (0.25,0.55);
\draw[line width = 0.06cm] (0.75,0.55) -- (1,0.55);
\draw[->, line width = 0.06cm] (0.5,0.8) -- (0.5,1.1);
\draw[->, line width = 0.06cm, white] (0.5,0.3) -- (0.5,0);
\end{tikzpicture}}
\end{picture}
}

\newcommand{\sdv[1]}{
\begin{picture}(1.5,0.7)
\put(0,-0.4){
\begin{tikzpicture}[scale=0.7]
\fill[black!35!white] (0,0.3) .. controls (0.5,0.3) and (0.5,0.3).. (1,0.3) ..  controls (1,0.55) and (1,0.55) .. (1,0.8)
 ..  controls (0.5,0.8) and (0.5,0.8) .. (0,0.8) ;
\draw[line width = 0.06cm] (0,0.3) -- (0,0.8);
\draw[line width = 0.06cm] (1,0.3) -- (1,0.8);
\put(0.25,0.33){$#1$};
\draw[line width = 0.06cm] (0,0.55) -- (0.25,0.55);
\draw[line width = 0.06cm] (0.75,0.55) -- (1,0.55);
\draw[->, line width = 0.06cm, white] (0.5,0.8) -- (0.5,1.1);
\draw[->, line width = 0.06cm] (0.5,0.3) -- (0.5,0);
\end{tikzpicture}}
\end{picture}
}

\newcommand{\uid}{
\begin{tikzpicture}[scale=0.7]
\draw[line width = 0.06cm] (0,-0.01) -- (0,2.02);
\draw[line width = 0.06cm] (1,-0.01) -- (1,2.02);
\end{tikzpicture}}

\newcommand{\puuv[1]}{
\begin{tikzpicture}[scale=0.7]
\draw[line width = 0.06cm] (0,0.3) -- (0,0.8);
\draw[line width = 0.06cm] (1,0.3) -- (1,0.8);
\put(0.25,0.33){$#1$};
\draw[line width = 0.06cm] (0,0.55) -- (0.25,0.55);
\draw[line width = 0.06cm] (0.75,0.55) -- (1,0.55);
\draw[->, line width = 0.06cm] (0.5,0.8) -- (0.5,1.1);
\draw[->, line width = 0.06cm, white] (0.5,0.3) -- (0.5,0);
\end{tikzpicture}
}

\newcommand{\pudv[1]}{
\begin{tikzpicture}[scale=0.7]
\draw[line width = 0.06cm] (0,0.3) -- (0,0.8);
\draw[line width = 0.06cm] (1,0.3) -- (1,0.8);
\put(0.25,0.33){$#1$};
\draw[line width = 0.06cm] (0,0.55) -- (0.25,0.55);
\draw[line width = 0.06cm] (0.75,0.55) -- (1,0.55);
\draw[->, line width = 0.06cm, white] (0.5,0.8) -- (0.5,1.1);
\draw[->, line width = 0.06cm] (0.5,0.3) -- (0.5,0);
\end{tikzpicture}
}

\newcommand{\psuv[1]}{
\begin{tikzpicture}[scale=0.7]
\fill[black!35!white] (0,0.3) .. controls (0.5,0.3) and (0.5,0.3).. (1,0.3) ..  controls (1,0.55) and (1,0.55) .. (1,0.8)
 ..  controls (0.5,0.8) and (0.5,0.8) .. (0,0.8) ;
\draw[line width = 0.06cm] (0,0.3) -- (0,0.8);
\draw[line width = 0.06cm] (1,0.3) -- (1,0.8);
\put(0.25,0.33){$#1$};
\draw[line width = 0.06cm] (0,0.55) -- (0.25,0.55);
\draw[line width = 0.06cm] (0.75,0.55) -- (1,0.55);
\draw[->, line width = 0.06cm] (0.5,0.8) -- (0.5,1.1);
\draw[->, line width = 0.06cm, white] (0.5,0.3) -- (0.5,0);
\end{tikzpicture}
}

\newcommand{\psdv[1]}{
\begin{tikzpicture}[scale=0.7]
\fill[black!35!white] (0,0.3) .. controls (0.5,0.3) and (0.5,0.3).. (1,0.3) ..  controls (1,0.55) and (1,0.55) .. (1,0.8)
 ..  controls (0.5,0.8) and (0.5,0.8) .. (0,0.8) ;
\draw[line width = 0.06cm] (0,0.3) -- (0,0.8);
\draw[line width = 0.06cm] (1,0.3) -- (1,0.8);
\put(0.25,0.33){$#1$};
\draw[line width = 0.06cm] (0,0.55) -- (0.25,0.55);
\draw[line width = 0.06cm] (0.75,0.55) -- (1,0.55);
\draw[->, line width = 0.06cm, white] (0.5,0.8) -- (0.5,1.1);
\draw[->, line width = 0.06cm] (0.5,0.3) -- (0.5,0);
\end{tikzpicture}
}

\newcommand{\Dtst}{
 \begin{tikzpicture}[scale=0.7]
\fill[black!35!white] (0,3.01) .. controls (0,2) and (0.5,2).. (1.5,1.5) ..  controls (1,2) and (1,2.5) .. (1,3.01);
\fill[black!35!white] (3,3.01) .. controls (3,2) and (2.5,2) .. (1.5,1.5) .. controls (2,2) and (2,2.5) .. (2,3.01);
\draw[line width = 0.07cm, white] (1.1,1.69) -- (1.9, 1.45);
\draw[line width = 0.06cm] (0,3) .. controls (0,2) and (0.5,2) .. (1.5,1.5) .. controls (2.5,1) and (3,1) .. (3,0);
\draw[line width = 0.08cm, white] (1.1,1.69) -- (1.9, 1.31);
\draw[line width = 0.06cm] (1,3) .. controls (1,2.5) and (1,2) .. (1.5,1.5) .. controls (2,1) and (2,0.5) .. (2,0);
\draw[line width = 0.08cm, white] (1.35,1.66) -- (1.65, 1.33);
\draw[line width = 0.06cm] (2,3) .. controls (2,2.5) and (2,2) .. (1.5,1.5) .. controls (1,1) and (1,0.5) .. (1,0);
\draw[line width = 0.07cm, white] (1.625,1.654) -- (1.383, 1.355);
\draw[line width = 0.06cm] (3,3) .. controls (3,2) and (2.5,2) .. (1.5,1.5) .. controls (0.5,1) and (0,1) .. (0,0);
  \end{tikzpicture}
}

\newcommand{\ShrunkDtst}{
 \begin{tikzpicture}[scale=0.7]
\fill[black!35!white] (0,1.51) .. controls (0,1) and (0.5,1).. (1.5,.75) ..  controls (1,1) and (1,1.25) .. (1,1.51);
\fill[black!35!white] (3,1.51) .. controls (3,1) and (2.5,1) .. (1.5,.75) .. controls (2,1) and (2,1.25) .. (2,1.51);
\draw[line width = 0.07cm, white] (1.1, .845) -- (1.9,.725);
\draw[line width = 0.06cm] (0,1.5) .. controls (0,1) and (0.5,1) .. (1.5,.75) .. controls (2.5,.5) and (3,.5) .. (3,0);
\draw[line width = 0.08cm, white] (1.1,.845) -- (1.9, .655);
\draw[line width = 0.06cm] (1,1.5) .. controls (1,1.25) and (1,1) .. (1.5,.75) .. controls (2,.5) and (2,0.25) .. (2,0);
\draw[line width = 0.08cm, white] (1.35,.83) -- (1.65, .66);
\draw[line width = 0.06cm] (2,1.5) .. controls (2,1.25) and (2,1) .. (1.5,.75) .. controls (1,.5) and (1,0.5) .. (1,0);
\draw[line width = 0.07cm, white] (1.625,.827) -- (1.383, .677);
\draw[line width = 0.06cm] (3,1.5) .. controls (3,1) and (2.5,1) .. (1.5,.75) .. controls (0.5,.5) and (0,.5) .. (0,0);
  \end{tikzpicture}
}

\newcommand{\twist}{
\begin{picture}(1.5,1.4)
\put(0,-0.8){
 \begin{tikzpicture}[scale=0.7]
\fill[black!35!white] (0,-0.01) .. controls (0,0.5) and (0,0.5) .. (0.5,1) .. controls (1,0.5) and (1,0.5) .. (1,-0.01);
\draw[line width = 0.06cm] (0,-0.01) .. controls (0,0.5) and (0,0.5) .. (0.5,1) .. controls (1,1.5) and (1,1.5) .. (1,2);
\draw[line width = 0.07cm, white] (0.4,0.9) .. controls (0.5,1) and (0.5,1) .. (0.6,1.1);
\draw[line width=0.06cm] (1,-0.01) .. controls (1,0.5) and (1,0.5) .. (0.5,1) .. controls (0,1.5) and (0,1.5) .. (0,2);
  \end{tikzpicture}}
  \end{picture}}

    \newcommand{\sultwist}{
\begin{picture}(1.5,1.4)
\put(0,-0.8){
 \begin{tikzpicture}[scale=0.7, xscale=1]
\fill[black!35!white] (0,-0.01) .. controls (0,0.5) and (0,0.5) .. (0.5,1) .. controls (1,0.5) and (1,0.5) .. (1,-0.01);
\draw[line width = 0.06cm] (0,-0.01) .. controls (0,0.5) and (0,0.5) .. (0.5,1) .. controls (1,1.5) and (1,1.5) .. (1,2);
\draw[line width = 0.07cm, white] (0.4,0.9) .. controls (0.5,1) and (0.5,1) .. (0.6,1.1);
\draw[line width=0.06cm] (1,-0.01) .. controls (1,0.5) and (1,0.5) .. (0.5,1) .. controls (0,1.5) and (0,1.5) .. (0,2);
  \end{tikzpicture}}
  \end{picture}}
  
\newcommand{\usltwist}{
\begin{picture}(1.5,1.4)
\put(0,-0.8){
 \begin{tikzpicture}[rotate=180, scale=0.7]
\fill[black!35!white] (0,-0.01) .. controls (0,0.5) and (0,0.5) .. (0.5,1) .. controls (1,0.5) and (1,0.5) .. (1,-0.01);
\draw[line width = 0.06cm] (0,-0.01) .. controls (0,0.5) and (0,0.5) .. (0.5,1) .. controls (1,1.5) and (1,1.5) .. (1,2);
\draw[line width = 0.07cm, white] (0.4,0.9) .. controls (0.5,1) and (0.5,1) .. (0.6,1.1);
\draw[line width=0.06cm] (1,-0.01) .. controls (1,0.5) and (1,0.5) .. (0.5,1) .. controls (0,1.5) and (0,1.5) .. (0,2);
  \end{tikzpicture}}
  \end{picture}}
  
  \newcommand{\usrtwist}{
\begin{picture}(1.5,1.4)
\put(0,-0.8){
 \begin{tikzpicture}[rotate=180, scale=0.7, xscale=-1]
\fill[black!35!white] (0,-0.01) .. controls (0,0.5) and (0,0.5) .. (0.5,1) .. controls (1,0.5) and (1,0.5) .. (1,-0.01);
\draw[line width = 0.06cm] (0,-0.01) .. controls (0,0.5) and (0,0.5) .. (0.5,1) .. controls (1,1.5) and (1,1.5) .. (1,2);
\draw[line width = 0.07cm, white] (0.4,0.9) .. controls (0.5,1) and (0.5,1) .. (0.6,1.1);
\draw[line width=0.06cm] (1,-0.01) .. controls (1,0.5) and (1,0.5) .. (0.5,1) .. controls (0,1.5) and (0,1.5) .. (0,2);
  \end{tikzpicture}}
  \end{picture}}
  
\newcommand{\ptwist}{
 \begin{tikzpicture}[scale=0.7]
\fill[black!35!white] (0,-0.01) .. controls (0,0.5) and (0,0.5) .. (0.5,1) .. controls (1,0.5) and (1,0.5) .. (1,-0.01);
\draw[line width = 0.06cm] (0,-0.01) .. controls (0,0.5) and (0,0.5) .. (0.5,1) .. controls (1,1.5) and (1,1.5) .. (1,2);
\draw[line width = 0.07cm, white] (0.4,0.9) .. controls (0.5,1) and (0.5,1) .. (0.6,1.1);
\draw[line width=0.06cm] (1,-0.01) .. controls (1,0.5) and (1,0.5) .. (0.5,1) .. controls (0,1.5) and (0,1.5) .. (0,2);
  \end{tikzpicture}}
  
  \newcommand{\ptwistusb}{
 \begin{tikzpicture}[scale=0.7, yscale=-1]
\fill[black!35!white] (0,-0.01) .. controls (0,0.5) and (0,0.5) .. (0.5,1) .. controls (1,0.5) and (1,0.5) .. (1,-0.01);
\draw[line width = 0.06cm] (0,-0.01) .. controls (0,0.5) and (0,0.5) .. (0.5,1) .. controls (1,1.5) and (1,1.5) .. (1,2);
\draw[line width = 0.07cm, white] (0.4,0.9) .. controls (0.5,1) and (0.5,1) .. (0.6,1.1);
\draw[line width=0.06cm] (1,-0.01) .. controls (1,0.5) and (1,0.5) .. (0.5,1) .. controls (0,1.5) and (0,1.5) .. (0,2);
  \end{tikzpicture}}

    \newcommand{\ptwistsub}{
 \begin{tikzpicture}[scale=0.7, xscale=-1]
\fill[black!35!white] (0,-0.01) .. controls (0,0.5) and (0,0.5) .. (0.5,1) .. controls (1,0.5) and (1,0.5) .. (1,-0.01);
\draw[line width = 0.06cm] (0,-0.01) .. controls (0,0.5) and (0,0.5) .. (0.5,1) .. controls (1,1.5) and (1,1.5) .. (1,2);
\draw[line width = 0.07cm, white] (0.4,0.9) .. controls (0.5,1) and (0.5,1) .. (0.6,1.1);
\draw[line width=0.06cm] (1,-0.01) .. controls (1,0.5) and (1,0.5) .. (0.5,1) .. controls (0,1.5) and (0,1.5) .. (0,2);
  \end{tikzpicture}}
  
  \newcommand{\pnegtwistusb}{
 \begin{tikzpicture}[scale=0.7, yscale=-1]
\fill[black!35!white] (0,-0.01) .. controls (0,0.5) and (0,0.5) .. (0.5,1) .. controls (1,0.5) and (1,0.5) .. (1,-0.01);
\draw[line width=0.06cm] (1,-0.01) .. controls (1,0.5) and (1,0.5) .. (0.5,1) .. controls (0,1.5) and (0,1.5) .. (0,2);
\draw[line width = 0.07cm, white] (0.4,1.1) .. controls (0.5,1) and (0.5,1) .. (0.6,.9);
\draw[line width = 0.06cm] (0,-0.01) .. controls (0,0.5) and (0,0.5) .. (0.5,1) .. controls (1,1.5) and (1,1.5) .. (1,2);
  \end{tikzpicture}}

  \newcommand{\pnegtwistsub}{
 \begin{tikzpicture}[scale=0.7, xscale=-1]
\fill[black!35!white] (0,-0.01) .. controls (0,0.5) and (0,0.5) .. (0.5,1) .. controls (1,0.5) and (1,0.5) .. (1,-0.01);
\draw[line width=0.06cm] (1,-0.01) .. controls (1,0.5) and (1,0.5) .. (0.5,1) .. controls (0,1.5) and (0,1.5) .. (0,2);
\draw[line width = 0.07cm, white] (0.4,1.1) .. controls (0.5,1) and (0.5,1) .. (0.6,.9);
\draw[line width = 0.06cm] (0,-0.01) .. controls (0,0.5) and (0,0.5) .. (0.5,1) .. controls (1,1.5) and (1,1.5) .. (1,2);
  \end{tikzpicture}}

\newcommand{\stwist}{
 \begin{tikzpicture}[rotate=180, scale=0.3]
\fill[black!35!white] (0,-0.01) .. controls (0,0.5) and (0,0.5) .. (0.5,1) .. controls (1,0.5) and (1,0.5) .. (1,-0.01);

\draw[line width=0.03cm] (1,-0.01) .. controls (1,0.5) and (1,0.5) .. (0.5,1) .. controls (0,1.5) and (0,1.5) .. (0,2);
\draw[line width = 0.04cm, white] (0.6,0.9) .. controls (0.5,1) and (0.5,1) .. (0.4,1.1);
\draw[line width = 0.03cm] (0,-0.01) .. controls (0,0.5) and (0,0.5) .. (0.5,1) .. controls (1,1.5) and (1,1.5) .. (1,2);
  \end{tikzpicture}}
  
  \newcommand{\sstwist}{
 \begin{tikzpicture}[rotate=180, scale=0.3, yscale=-1, xscale=-1]
\fill[black!35!white] (0,-0.01) .. controls (0,0.5) and (0,0.5) .. (0.5,1) .. controls (1,0.5) and (1,0.5) .. (1,-0.01);

\draw[line width=0.03cm] (1,-0.01) .. controls (1,0.5) and (1,0.5) .. (0.5,1) .. controls (0,1.5) and (0,1.5) .. (0,2);
\draw[line width = 0.04cm, white] (0.6,0.9) .. controls (0.5,1) and (0.5,1) .. (0.4,1.1);
\draw[line width = 0.03cm] (0,-0.01) .. controls (0,0.5) and (0,0.5) .. (0.5,1) .. controls (1,1.5) and (1,1.5) .. (1,2);
  \end{tikzpicture}}

\newcommand{\ucap}{
\begin{tikzpicture}[scale=0.7]
\draw[line width = 0.06cm]  (0,0) arc (180:0:1.5cm);
\draw[line width = 0.06cm]  (1,0) arc (180:0:0.5cm);
\end{tikzpicture}}

\newcommand{\uudcap[1]}{
\setlength{\unitlength}{0.7cm}
\begin{picture}(3.4,1.4)
\put(0,-0.25){\ucap}
\put(0,-1){\puuv[#1]}
\put(2,-1){\pudv[#1]}
\end{picture}
}

\newcommand{\uducap[1]}{
\setlength{\unitlength}{0.7cm}
\begin{picture}(3.4,1.4)
\put(0,-0.25){\ucap}
\put(0,-1){\pudv[#1]}
\put(2,-1){\puuv[#1]}
\end{picture}
}

\newcommand{\sducap[1]}{
\setlength{\unitlength}{0.7cm}
\begin{picture}(3.4,1.4)
\put(0,-1){\psdv[#1]}
\put(0,-0.25){\scap}
\put(2,-1){\psuv[#1]}
\end{picture}
}

\newcommand{\sudcap[1]}{
\setlength{\unitlength}{0.7cm}
\begin{picture}(3.4,1.4)
\put(2,-1){\psdv[#1]}
\put(0,-0.25){\scap}
\put(0,-1){\psuv[#1]}
\end{picture}
}

\newcommand{\ucup}{
\begin{tikzpicture}[scale=0.7, rotate=180]
\draw[line width = 0.06cm]  (0,0) arc (180:0:1.5cm);
\draw[line width = 0.06cm]  (1,0) arc (180:0:0.5cm);
\end{tikzpicture}}

\newcommand{\uudcup[1]}{
\setlength{\unitlength}{0.7cm}
\begin{picture}(3.4,1.4)
\put(-0.15,-0.75){
\begin{picture}(3.4,1.4)
\put(0,-0.25){\ucup}
\put(0,0.9){\puuv[#1]}
\put(2,0.9){\pudv[#1]}
\end{picture}}
\end{picture}
}

\newcommand{\uducup[1]}{
\setlength{\unitlength}{0.7cm}
\begin{picture}(3.4,1.4)
\put(-0.15,-0.75){
\begin{picture}(3.4,1.4)
\put(0,-0.25){\ucup}
\put(0,0.9){\pudv[#1]}
\put(2,0.9){\puuv[#1]}
\end{picture}}
\end{picture}
}

\newcommand{\sducup[1]}{
\setlength{\unitlength}{0.7cm}
\begin{picture}(3.4,1.4)
\put(-0.15,-0.75){
\begin{picture}(3.4,1.4)
\put(2,0.9){\psuv[#1]}
\put(0,-0.25){\scup}
\put(0,0.9){\psdv[#1]}
\end{picture}}
\end{picture}
}

\newcommand{\sudcup[1]}{
\setlength{\unitlength}{0.7cm}
\begin{picture}(3.4,1.4)
\put(-0.15,-0.75){
\begin{picture}(3.4,1.4)
\put(0,0.9){\psuv[#1]}
\put(0,-0.25){\scup}
\put(2,0.9){\psdv[#1]}
\end{picture}}
\end{picture}
}

\newcommand{\plaincrossing}{
\begin{picture}(3.4,1.3) \put(0,-1.45){
\begin{tikzpicture}[scale=0.5]
\draw[line width = 0.06cm] (0,3) .. controls (0,2) and (2,1) .. (2,0);
\draw[line width = 0.06cm] (1,3) .. controls (1,2) and (3,1) .. (3,0);
\draw[line width = 0.8cm, white] (0.5,0) .. controls (0.5,1) and (2.5,2) .. (2.5,3);
\draw[line width = 0.06cm] (0,0) .. controls (0,1) and (2,2) .. (2,3);
\draw[line width = 0.06cm] (1,0) .. controls (1,1) and (3,2) .. (3,3);
\end{tikzpicture}}
\end{picture}}

\newcommand{\pplaincrossing}{
\begin{tikzpicture}[scale=0.7]
\draw[line width = 0.06cm] (0,3) .. controls (0,2) and (2,1) .. (2,0);
\draw[line width = 0.06cm] (1,3) .. controls (1,2) and (3,1) .. (3,0);
\draw[line width = 0.8cm, white] (0.5,0) .. controls (0.5,1) and (2.5,2) .. (2.5,3);
\draw[line width = 0.06cm] (0,0) .. controls (0,1) and (2,2) .. (2,3);
\draw[line width = 0.06cm] (1,0) .. controls (1,1) and (3,2) .. (3,3);
\end{tikzpicture}}

\newcommand{\sucrossingisotopy}{
\begin{tikzpicture}[scale=0.4]
\fill[black!35!white] (0,3) .. controls (0,2) and (2,1) .. (2,0) .. controls (2.5,0) and (2.5,0) .. (3,0) .. controls (3,1) and (1,2) .. (1,3);

\fill[black!35!white] (0,3) -- (0,5) -- (1,5) -- (1,3);

\fill[black!35!white] (2,-2) -- (2,0.1) -- (3,0.1) -- (3,-2);

\draw[line width = 0.06cm] (0,3) .. controls (0,2) and (2,1) .. (2,0);
\draw[line width = 0.06cm] (1,3) .. controls (1,2) and (3,1) .. (3,0);
\draw[line width = 0.5cm, white] (0.5,0) .. controls (0.5,1) and (2.5,2) .. (2.5,3);
\draw[line width = 0.06cm] (0,0) .. controls (0,1) and (2,2) .. (2,3);
\draw[line width = 0.06cm] (1,0) .. controls (1,1) and (3,2) .. (3,3);

\draw[line width = 0.06cm] (0,-2.01) -- (0,0.02);
\draw[line width = 0.06cm] (1,-2.01) -- (1,0.02);

\draw[line width = 0.06cm] (2,3.01) -- (2,5.02);
\draw[line width = 0.06cm] (3,3.01) -- (3,5.02);

\draw[line width = 0.06cm] (2,-2.01) -- (2,0.02);
\draw[line width = 0.06cm] (3,-2.01) -- (3,0.02);

\draw[line width = 0.06cm] (0,3.01) -- (0,5.02);
\draw[line width = 0.06cm] (1,3.01) -- (1,5.02);

\draw[line width = 0.06cm] (6,3) .. controls (6,2) and (8,1) .. (8,0);
\draw[line width = 0.06cm] (7,3) .. controls (7,2) and (9,1) .. (9,0);
\draw[line width = 0.5cm, white] (6.5,0) .. controls (6.5,1) and (8.5,2) .. (8.5,3);
\draw[line width = 0.06cm] (6,0) .. controls (6,1) and (8,2) .. (8,3);
\draw[line width = 0.06cm] (7,0) .. controls (7,1) and (9,2) .. (9,3);

\draw[line width = 0.06cm] (6,-2.01) -- (6,0.02);
\draw[line width = 0.06cm] (7,-2.01) -- (7,0.02);
\draw[line width = 0.06cm] (8,3.01) -- (8,5.02);
\draw[line width = 0.06cm] (9,3.01) -- (9,5.02);

\fill[black!35!white] (8,-2.01) .. controls (8,-1.5) and (8,-1.5) .. (8.5,-1) .. controls (9,-1.5) and (9,-1.5) .. (9,-2.01);
\draw[line width = 0.06cm] (8,-2.01) .. controls (8,-1.5) and (8,-1.5) .. (8.5,-1) .. controls (9,-0.5) and (9,-0.5) .. (9,0);
\draw[line width = 0.07cm, white] (8.4,-1.1) .. controls (8.5,-1) and (8.5,-1) .. (8.6,-0.9);
\draw[line width=0.06cm] (9,-2.01) .. controls (9,-1.5) and (9,-1.5) .. (8.5,-1) .. controls (8,-0.5) and (8,-0.5) .. (8,0);

\fill[black!35!white] (6,5.01) .. controls (6,4.5) and (6,4.5) .. (6.5,4) .. controls (7,4.5) and (7,4.5) .. (7,5.01);
\draw[line width = 0.06cm] (6,5.01) .. controls (6,4.5) and (6,4.5) .. (6.5,4) .. controls (7,3.5) and (7,3.5) .. (7,3);
\draw[line width = 0.07cm, white] (6.4,4.1) .. controls (6.5,4) and (6.5,4) .. (6.6,3.9);
\draw[line width=0.06cm] (7,5.01) .. controls (7,4.5) and (7,4.5) .. (6.5,4) .. controls (6,3.5) and (6,3.5) .. (6,3);

\draw(4.4,1.8) node {$\simeq$};

\end{tikzpicture}
}

\newcommand{\dasigma}{
\setlength{\unitlength}{0.3cm}
\begin{tikzpicture}[scale=0.3, xscale=-1]

\draw[line width = 0.03cm] (7,0) .. controls (7,1) and (9,4) .. (9,5);
\draw[line width = 0.03cm] (8,0) .. controls (8,1) and (10,4) .. (10,5);
\draw[line width = 0.03cm] (10,0) .. controls (10,1) and (12,4) .. (12,5);
\draw[line width = 0.03cm] (11,0) .. controls (11,1) and (13,4) .. (13,5);

\draw[line width = 0.3cm, white] (7.5,5) .. controls (7.5,4) and (12.5,1) .. (12.5,0);
\draw[line width = 0.03cm] (7,5) .. controls (7,4) and (12,1) .. (12,0);
\draw[line width = 0.03cm] (8,5) .. controls (8,4) and (13,1) .. (13,0);

\put(-11.5,4.8){$. . .$};

\put(-9.5,0){$. . .$};
\setlength{\unitlength}{0.7cm}
\end{tikzpicture}}

\newcommand{\longellement}{
\setlength{\unitlength}{0.3cm}
\begin{tikzpicture}[scale=0.3, yscale=-1]

\draw[line width = 0.3cm, white] (7.5,5) .. controls (7.5,4) and (0.5,1) .. (0.5,0);
\draw[line width = 0.03cm] (8,5) .. controls (8,4) and (1,1) .. (1,0);
\draw[line width = 0.03cm] (7,5) .. controls (7,4) and (0,1) .. (0,0);

\draw[line width = 0.3cm, white] (5.5,5) .. controls (5.5,4) and (2.5,1) .. (2.5,0);
\draw[line width = 0.03cm] (6,5) .. controls (6,4) and (3,1) .. (3,0);
\draw[line width = 0.03cm] (5,5) .. controls (5,4) and (2,1) .. (2,0);

\draw[line width = 0.32cm, white] (2.5,5) .. controls (2.5,4) and (5.5,1) .. (5.5,0);
\fill[black!35!white] (2,5) .. controls (2,4) and (5,1) .. (5,0) .. controls (5.5,0) and (5.5,0) .. (6,0) .. controls (6,1) and (3,4) ..  (3,5) .. controls (2.5,5) and (2.5,5) .. (2,5);
\draw[line width = 0.03cm] (2,5) .. controls (2,4) and (5,1) .. (5,0);
\draw[line width = 0.03cm] (3,5) .. controls (3,4) and (6,1) .. (6,0);

\draw[line width = 0.3cm, white] (0.5,5) .. controls (0.5,4) and (7.5,1) .. (7.5,0);
\draw[line width = 0.03cm] (0,5) .. controls (0,4) and (7,1) .. (7,0);
\draw[line width = 0.03cm] (1,5) .. controls (1,4) and (8,1) .. (8,0);

\put(3.5,-0.2){$. . .$};
\put(3.5,-5){$. . .$};

\setlength{\unitlength}{0.7cm}
\end{tikzpicture}}

\newcommand{\ttopns}{
\begin{tikzpicture}[scale=0.3, xscale=-1]

\fill[black!35!white] (1,5) .. controls (1,4) and (7,1) .. (7,0) .. controls (7.5,0) and (7.5,0) .. (8,0) .. controls (8,1) and (0,4) .. (0,5);
\fill[black!35!white] (2,5) .. controls (2,4) and (6,1) .. (6,0) .. controls (5.5,0) and (5.5,0) .. (5,0) .. controls (5,1) and (3,4) .. (3,5);
\fill[black!35!white] (6,5) .. controls (6,4) and (2,1) .. (2,0).. controls (2.5,0) and (2.5,0) .. (3,0) .. controls (3,1) and (5,4) .. (5,5);
\fill[black!35!white] (8,5) .. controls (8,4) and (0,1) .. (0,0) .. controls (0.5,0) and (0.5,0) .. (1,0) .. controls (1,1) and (7,4) .. (7,5);

\fill[white] (0,0) .. controls (0,1) and (0,1) .. (0,2.5) .. controls (4,2.5) and (4,2.5) .. (8,2.5) .. controls (8,1) and (8,1) ..(8,0);

\fill[black!35!white] (2,5) .. controls (2,4) and (6,1) .. (6,0) .. controls (5.5,0) and (5.5,0) .. (5,0) .. controls (5,1) and (3,4) .. (3,5);

\fill[black!0!white] (2,5) .. controls (2,4.5) and (2,4.5) .. (4,2.5) .. controls (3.5,3.5) and (3.5,3.5).. (3,5);

\draw[line width = 0.03cm] (8,5) .. controls (8,4) and (0,1) .. (0,0);
\draw[line width = 0.03cm] (7,5) .. controls (7,4) and (1,1) .. (1,0);

\draw[line width = 0.08cm, white] (3.56,2) -- (4.44,3.05);
\fill[black!35!white] (4,2.5) -- (4.4,3.3) -- (4.6,3.1);

\draw[line width = 0.03cm] (6,5) .. controls (6,4) and (2,1) .. (2,0);
\draw[line width = 0.03cm] (5,5) .. controls (5,4) and (3,1) .. (3,0);
\draw[line width = 0.03cm] (2,5) .. controls (2,4) and (6,1) .. (6,0);
\draw[line width = 0.03cm] (3,5) .. controls (3,4) and (5,1) .. (5,0);

\draw[line width = 0.12cm, white] (3.3,2.85) -- (4.7,2.15);
\draw[line width = 0.03cm] (1,5) .. controls (1,4) and (7,1) .. (7,0);

\draw[line width = 0.03cm] (0,5) .. controls (0,4) and (8,1) .. (8,0);

\setlength{\unitlength}{0.7cm}
\end{tikzpicture}}

\newcommand{\makeit}{
\setlength{\unitlength}{0.3cm}
\begin{picture}(4,8)
\put(0,0){\longellement}
\put(0.4335,5.46){\stwist}
\put(2.4335,5.46){\stwist}
\put(5.4335,5.46){\sstwist}
\put(7.4335,5.46){\stwist}
\end{picture} 
\setlength{\unitlength}{0.7cm}
}

\newcommand{\ssid}{
\setlength{\unitlength}{0.3cm}
\begin{picture}(1,2)
\put(0,0){\begin{tikzpicture}[scale=0.3]
\fill[black!35!white] (0,0) -- (0,2) -- (1,2) -- (1,0);
\draw[line width = 0.03cm] (0,0) -- (0,2);
\draw[line width = 0.03cm] (1,0) -- (1,2);
\end{tikzpicture}}
\end{picture}
\setlength{\unitlength}{0.7cm}
}

\newcommand{\suid}{
\setlength{\unitlength}{0.3cm}
\begin{picture}(1,2)
\put(0,0){\begin{tikzpicture}[scale=0.3]
\draw[line width = 0.03cm] (0,0) -- (0,2);
\draw[line width = 0.03cm] (1,0) -- (1,2);
\end{tikzpicture}}
\end{picture}
\setlength{\unitlength}{0.7cm}
}
 
\newcommand{\Ltst}{
 \begin{tikzpicture}[scale=0.7]
\fill[black!35!white] (0,0) .. controls (0,1) and (0.5,1).. (1.5,1.5) .. controls (1,1) and (1, 0.5) .. (1,0);
\fill[black!35!white] (3,0) .. controls (3,1) and (2.5,1).. (1.5,1.5) .. controls (2,1) and (2,0.5) .. (2,0);
\draw[line width = 0.07cm, white] (1.1,1.69) -- (1.9, 1.45);
\draw[line width = 0.06cm] (0,3) .. controls (0,2) and (0.5,2) .. (1.5,1.5) .. controls (2.5,1) and (3,1) .. (3,0);
\draw[line width = 0.08cm, white] (1.1,1.69) -- (1.9, 1.31);
\draw[line width = 0.06cm] (1,3) .. controls (1,2.5) and (1,2) .. (1.5,1.5) .. controls (2,1) and (2,0.5) .. (2,0);
\draw[line width = 0.08cm, white] (1.35,1.66) -- (1.65, 1.33);
\draw[line width = 0.06cm] (2,3) .. controls (2,2.5) and (2,2) .. (1.5,1.5) .. controls (1,1) and (1,0.5) .. (1,0);
\draw[line width = 0.07cm, white] (1.625,1.654) -- (1.383, 1.355);
\draw[line width = 0.06cm] (3,3) .. controls (3,2) and (2.5,2) .. (1.5,1.5) .. controls (0.5,1) and (0,1) .. (0,0);
  \end{tikzpicture}
}

\newcommand{\ShrunkLtst}{
 \begin{tikzpicture}[scale=0.7]
\fill[black!35!white] (0,0) .. controls (0,.5) and (0.5,.5).. (1.5,.75) .. controls (1,.5) and (1, 0.25) .. (1,0);
\fill[black!35!white] (3,0) .. controls (3,.5) and (2.5,.5).. (1.5,.75) .. controls (2,.5) and (2,0.25) .. (2,0);
\draw[line width = 0.07cm, white] (1.1,.845) -- (1.9, .725);
\draw[line width = 0.06cm] (0,1.5) .. controls (0,1) and (0.5,1) .. (1.5,.75) .. controls (2.5,.5) and (3,.5) .. (3,0);
\draw[line width = 0.08cm, white] (1.1,.845) -- (1.9, .655);
\draw[line width = 0.06cm] (1,1.5) .. controls (1,1.25) and (1,1) .. (1.5,.75) .. controls (2,.5) and (2,0.25) .. (2,0);
\draw[line width = 0.08cm, white] (1.35,.83) -- (1.65, .665);
\draw[line width = 0.06cm] (2,1.5) .. controls (2,1.25) and (2,1) .. (1.5,.75) .. controls (1,.5) and (1,0.25) .. (1,0);
\draw[line width = 0.07cm, white] (1.625,.827) -- (1.383, .6775);
\draw[line width = 0.06cm] (3,1.5) .. controls (3,1) and (2.5,1) .. (1.5,.75) .. controls (0.5,.5) and (0,.5) .. (0,0);
  \end{tikzpicture}
}

\newtheorem{Theorem}{Theorem}[section]
\newtheorem{Proposition}[Theorem]{Proposition} 
\newtheorem{Lemma}[Theorem]{Lemma}

\newtheorem{Definition}[Theorem]{Definition}

\newtheorem{Question}{Question}

\theoremstyle{definition}
\newtheorem{Comment}[Theorem]{Comment}

\def\fsl{\mathfrak{sl}}

\newcommand{\low}{\text{low}}

\newcommand{\wt}{\operatorname{wt}}

\newcommand{\bz}{\mathbb{Z}}
\newcommand{\bc}{\mathbb{C}}
\newcommand{\br}{\mathbb{R}}
\newcommand{\g}{\mathfrak{g}}

\newcommand{\End}{\mbox{End}}

\newcommand{\Flip}{\mbox{Flip}}

\newcommand{\THRIB}{\mathcal{HRIB}}
\newcommand{\HRIB}{\mathcal{HRIB}}

\newcommand{\Ribbon}{\mathcal{RIBBON}}
\newcommand{\DR}{\mathcal{RIBBON}}
\newcommand{\Hopf}{\mathcal{H}}

\newcommand{\cF}{\mathcal{F}}
\newcommand{\rev}{\text{rev}}

\newcommand{\ev}{\text{ev}}
\newcommand{\coev}{\text{coev}}
\newcommand{\qtr}{\widetilde{\text{ev}}}
\newcommand{\coqtr}{\widetilde{\text{coev}}}
\newcommand{\id}{\text{id}}
\newcommand{\Hom}{\text{Hom}}

\theoremstyle{definition}

\begin{document}

\title[The half-twist]{The half-twist for $U_q(\g)$ representations}

\author{Noah Snyder}
\email{nsnyder@math.berkeley.edu}
\address{UC Berkeley, Department of Mathematics\\ Berkeley, CA}

\author{Peter Tingley}
\email{ptingley@math.mit.edu}
\address{Peter Tingley,
MIT dept. of math,
77 Massachusetts Ave,
Cambridge, MA, 02139}

\begin{abstract}

We introduce the notion of a half-ribbon Hopf algebra, which is a Hopf algebra $\Hopf$ along with a distinguished element $t \in \Hopf$ such that $(\Hopf,  R,C)$ is a ribbon Hopf algebra, where $R= (t^{-1} \otimes t^{-1})\Delta(t)$ and $C= t^{-2}$. The element $t$ is closely related to the topological `half-twist', which twists a ribbon by 180 degrees.  We construct a functor from a topological category of ribbons with half-twists to the category of representations of any half-ribbon Hopf algebra.  We show that $U_q(\g)$ is a (topological) half-ribbon Hopf algebra, but that $t^{-2}$ is not the standard ribbon element.  For $U_q(\fsl_2)$, we show that there is no half-ribbon element $t$ such that $t^{-2}$ is the standard ribbon element. We then discuss how ribbon elements can be modified, and some consequences of these modifications. \end{abstract}

\maketitle

\tableofcontents

\section{Introduction}

Let $\DR(S)$ be the category whose morphisms consist of tangles of oriented directed ribbons up to isotopy, each labeled with an element of some set $S$. There is a notion of a ribbon Hopf algebra $\Hopf$ (see for example \cite{CP}), which is related to this topological category by the fact that there is a monoidal functor $\cF'$ from $\DR(\Hopf \text{-rep})$ to the category of representations $\Hopf\text{-rep}$. This allows one to construct invariants of oriented framed links, and from there invariants of ordinary links. 

There is a morphism in $\DR(S)$ which twists a ribbon by 360 degrees, but not one that twists a ribbon by 180 degrees (since negatively oriented objects are not allowed). We propose that one should consider a slightly larger category, denoted $\THRIB(S)$, where this 180 degree twist, the ``half-twist," is allowed. The half-twist can be applied to several ribbons at once, and all morphisms in $\THRIB(S)$ can be constructed out of half-twists, along with various caps and cups.  The following isotopy shows how the crossing in $\THRIB(S)$ is constructed out of the half-twist:

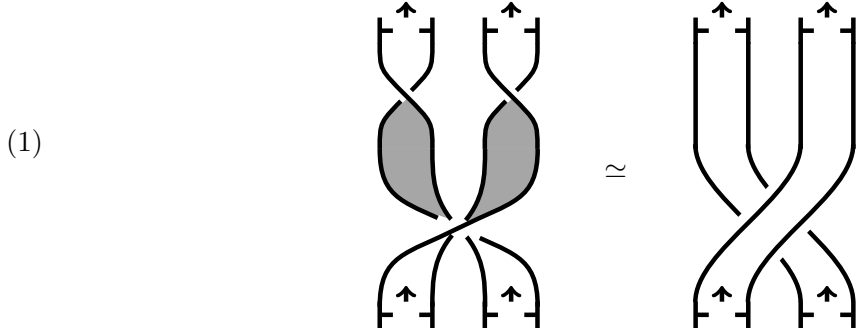
\setlength{\unitlength}{0.7cm}
\begin{equation} \label{main_isotopy}
\begin{picture}(9,2.1)
\put(0,-3){
\begin{picture}(9,6)
\put(0,0){\Dtst}
\put(0,3){\ptwist}
\put(2,3){\ptwist}
\put(0,-0.7){\puuv[U]}
\put(2,-0.7){\puuv[V]}
\put(0,4.7){\puuv[V]}
\put(2,4.7){\puuv[U]}

\put(4.5,2.5){$\simeq$}
\put(5.97,-0.44){\pplaincrossing}
\put(6,3){\uid}
\put(8,3){\uid}
\put(6,-0.7){\puuv[U]}
\put(8,-0.7){\puuv[V]}
\put(6,4.7){\puuv[V]}
\put(8,4.7){\puuv[U]}
\end{picture}}
\end{picture}
\end{equation}
\vspace{1.7cm}

Recall that a ribbon Hopf algebra $\Hopf$ is a Hopf algebra along with two extra features:

$\bullet$ A universal $R$-matrix $R \in \Hopf \otimes \Hopf$. The functor $\cF'$ takes a simple crossing of ribbons labeled $V$ and $W$ to $\Flip \circ R$ acting on $V \otimes W$.

$\bullet$ A central ``ribbon" element $C \in \Hopf$. The functor $\cF'$ takes a 360 degree twist of a ribbon labeled $V$ to $C$ acting on $V$.

\noindent The elements $R$ and $C$ must satisfy various compatibility conditions (see Section \ref{main1}). 

In the present work we define a half-ribbon Hopf algebra to be a ribbon Hopf algebra along with a distinguished element $t$ such that
\begin{enumerate}
\item \label{rtc} $R= (t^{-1} \otimes t^{-1}) \Delta(t)$ 
\item $C=t^{-2}$.
\end{enumerate}
We show that, if $\Hopf$ is a half-ribbon Hopf algebra, then $\cF'$ can be extended to a functor $\cF$ from $\HRIB(S)$ to $\Hopf\text{-rep}$.

Our main interest is the case where $\Hopf$ is the quantized universal enveloping algebra $U_q(\g)$ of a finite dimensional complex simple Lie algebra $\g$. In this case $U_q(\g)$ is actually a topological ribbon Hopf algebra, meaning that $R$ and $C$ only lie in some completion, not in $U_q(\g)$ itself. We define a topological half-ribbon Hopf algebra by allowing $t$ to lie in a completion of $\Hopf$, and show that $U_q(\g)$ has this structure. The main ingredient is a formula for the $R$ matrix of $U_q(\g)$, due to  Kirilov-Reshetikhin \cite{KR:1990} and Levendorskii-Soibelman \cite{LS}, of the form
\begin{equation} \label{XRint}
R= (X^{-1} \otimes X^{-1}) \Delta(X).
\end{equation}
The correspondence to condition (\ref{rtc}) above should be clear, and we show that $X$ is in fact a half-ribbon element for $U_q(\g)$. Interestingly, $X^{-2}$ is not the standard ribbon element, and in fact we show that is certain cases it is not possible to find a half-ribbon element $t$ for $U_q(\g)$ such that $t^{-2}$ is the standard ribbon element. 

In Section \ref{fiveaa} we discuss the different ribbon elements for $U_q(\g)$, and in some small cases we describe exactly which ones arise from half-ribbon elements. We then discuss some consequences of using $X^{-2}$ in place of the usual ribbon element. In particular, it simplifies the correspondence between certain skein theoretic constructions of link invariants and the quantum group constructions of the same objects, essentially by explaining some annoying negative signs that appear in, for example, \cite{Ohtsuki} or \cite{Kuperberg}.

We feel it would be of considerable interest to study which ribbon Hopf algebras can be given the structure of half-ribbon Hopf algebras. It would also be nice to give straightforward conditions on an element $t$ in a general Hopf algebra $\Hopf$, such that $\Hopf$ along with $t$ is a half-ribbon Hopf algebra (that is, $\Hopf$ is a ribbon Hopf algebra, with $R=(t^{-1} \otimes t^{-1}) \Delta(t)$ and  $C=t^{-2}$). At the end of the paper we discuss these and other possible future directions.

\subsection{Acknowledgments}

We would like to thank Joel Kamnitzer, Scott Morrison, and Nicolai Reshetikhin for helpful conversations.  We would also like to thank the anonymous referee for very thorough and helpful comments. Both authors were supported in part by RTG grant DMS-0354321. The second author was also partly supported by the Australian research council grant DP0879951 and NSF grant DMS-0902649.

\section{Conventions} \label{notation}
We first fix some notation. For the most part we follow conventions from \cite{CP}.

$\bullet$ $\g$ is a complex simple Lie algebra with Cartan algebra $ \mathfrak{h} $, and $A = (a_{ij})_{i,j \in I}$ is its Cartan matrix. 

$\bullet$ $ \langle \cdot , \cdot \rangle $ denotes the paring between $ \mathfrak{h} $ and $ \mathfrak{h}^\star $ and $ ( \cdot , \cdot) $ denotes the usual symmetric bilinear form on either $ \mathfrak{h}$ or $ \mathfrak{h}^\star $.  Fix the usual bases $ \alpha_i $ for $ \mathfrak{h}^\star $ and $ H_i $ for $\mathfrak{h}$, and recall that $ \langle H_i, \alpha_j \rangle = a_{ij} $.  

$\bullet$ $ d_i = (\alpha_i, \alpha_i)/2 $, so that $ (H_i, H_j) = d_j^{-1} a_{ij} $.  Let $ B $ denote the matrix $ (d_j^{-1} a_{ij}) $.

$\bullet$ $ q_i = q^{d_i} $.  

$\bullet$ $\rho$ is the element of $\mathfrak{h}^*$ such that $(\alpha_i, \rho)= d_i$ for all $ i $.

$\bullet$ $\rho^\vee$ is the element of $\mathfrak{h}$ such that $\langle \alpha_i, \rho^\vee \rangle=1$ for all $i$. 

$\bullet$ $s_i$ is the element of the Weyl group which is defined by $s_i(\alpha_j)=\alpha_j-\langle \alpha_i, \alpha_j^\vee \rangle \alpha_i$.

$\bullet$ $\theta$ is the diagram automorphism such that $w_0 (\alpha_i) = - \alpha_{\theta(i)},$ where $w_0$ is the longest element in the Weyl group.

$\bullet$ $U_q(\g)$ is the quantized universal enveloping algebra associated to $\g$, generated over $\mathbb{C}(q)$ by $E_i$, $F_i$ for all  $i \in I$, and $K_w$ for $w$ in the co-weight lattice of $\g$. As usual, let $K_i= K_{H_i}.$ We actually must adjoin a fixed $k^{th}$ root of $q$ to the base field, for some $k$ depending on $\g$. This causes no real difficulty, and for the most part we ignore it. For convenience, we recall the exact formula for the coproduct and antipode:
\begin{equation} \label{coproduct}
\begin{cases}
\Delta{E_i}  = E_i \otimes K_i + 1 \otimes E_i \\
\Delta{F_i} = F_i \otimes 1 + K_i^{-1} \otimes F_i \\ 
\Delta{K_i} = K_i \otimes K_i
\end{cases} 
\end{equation}
\begin{equation}
\begin{cases}
S(E_i)=-E_iK_i^{-1} \\
S(F_i)=-K_i F_i \\
S(K_i)=K_i^{-1} 
\end{cases}
\end{equation}

$\bullet$ $[n]= \frac{q^n - q^{-n}}{q-q^{-1}},$ and $X^{(n)} = \frac{X^n}{[n][n-1] \cdots [2]}.$

$\bullet$ $V_\lambda$ is the irreducible representation of $U_q(\g)$ with highest weight $\lambda$, and $v_\lambda$ is a highest weight vector. 

$\bullet$ $P$ is the weight lattice for $\g$ and $Q$ is the root lattice. 

$\bullet$ The standard ribbon element $C$ for $U_q(\g)$ acts on $V_\lambda$ as the constant $q^{-(\lambda, \lambda) - 2 (\lambda, \rho)}$.

$\bullet$ $U_q(\g)$-rep is the category of finite dimensional Type $1$ representations.

$\bullet$ $(\Hopf, R)$ is a quasi-triangular Hopf algebra over a field $F$, where 
\begin{enumerate}
\item $\mu: \Hopf \otimes \Hopf \rightarrow \Hopf$ is multiplication.
\item $\iota:  F \rightarrow \Hopf$ is the unit.
\item $\Delta: \Hopf \rightarrow \Hopf \otimes \Hopf$ is the comultiplication.
\item $\varepsilon: \Hopf \rightarrow F$ is the counit.
\item $S: \Hopf \rightarrow \Hopf$ is the antipode. 
\end{enumerate}

$\bullet$  $\Flip: V\otimes W \rightarrow W \otimes V$ is defined by $\Flip(v \otimes w) = w \otimes v$.  For longer tensor products, define $\rev: V_1 \otimes \cdots \otimes V_k \rightarrow V_k \otimes \cdots V_1$ by $v_1 \otimes \cdots \otimes v_k \mapsto v_k \otimes \cdots \otimes v_1$.

$\bullet$ If $V$ is a representation of a Hopf algebra $H$, then we define the left dual $V^*$ to be the dual vector space with the action $x \circ f (v) = f(S(x)v)$, and the right dual ${}^*V$ to be the dual vector space with the action $x \circ f (v) = f(S^{-1}(x)v)$.  Note that $({}^*V)^* = V = {}^*(V^*)$.

\section{Background}

Much of the motivation for this paper comes from studying an expression for the $R$-matrix 
\begin{equation} \label{Rmake_eq}
R=(X^{-1} \otimes X^{-1}) \Delta(X),
\end{equation}
where $X$ is an element in some completion of $U_q(\g)$.
This was first introduced by Kirilov-Reshetikhin \cite{KR:1990} and Levendorskii-Soibelman \cite{LS}, and has recently proven useful in studying the relationship between the braiding and crystal bases (see \cite{Rcommutor}). In this section we review this formula, and also recall the definition and basic properties of a ribbon Hopf algebra.

\subsection{A completion of $U_q(\g)$} \label{completion}

The element $X$ from (\ref{Rmake_eq}) is not actually in $U_q(\g)$, but only in a completion. In order to be precise, we briefly review the completion of $U_q(\g)$ that we use.

\begin{Definition}
$\widetilde{U_q(\g)}$ is the completion of $U_q(\g)$ in the weak topology generated by all matrix elements of all (finite-dimensional Type $1$) representations. Similarly, $\widetilde{U_q(\g) \otimes U_q(\g)}$ is the completion of $U_q(\g) \otimes U_q(\g)$ in the weak topology defined by all matrix elements of representations $V_\lambda \otimes V_\mu$, for all ordered pairs $(\lambda, \mu)$.
\end{Definition}

 \begin{Theorem} \label{directproduct} \cite{Rcommutor} $\widetilde{U_q(\g)}$ is the direct product of the endomorphism rings of all irreducible representations of $U_q(\g)$. That is:
 \begin{equation}
 \widetilde{U_q(\g)} =  \prod_{\lambda \in \Lambda_+} \End (V_\lambda).
 \end{equation}
 \end{Theorem}
 
 \begin{Comment}
 It is straightforward to see that $ \widetilde{U_q(\g)}$ is a topological Hopf algebra.
 \end{Comment}

By Theorem \ref{directproduct}, specifying an element of $\widetilde{U_q(\g)}$ is exactly the same as specifying how it acts on each isomorphism class of irreducible representation.  Similarly, if we want to specify an element of $\widetilde{U_q(\g)^{\otimes 2}}$ we just need to say how it acts on every tensor product of any two irreducible representations.

\subsection{A method of constructing commutativity constraints} \label{aut_int}

The following theorem is based on an idea used by Henriques and Kamnitzer to study the crystal commutor \cite{cactus}, and was further developed in \cite{Rcommutor}. The proof is a straightforward exercise, which can be found in \cite{Rcommutor}.

\begin{Proposition} (see \cite[Proposition 3.11]{Rcommutor}) \label{coanti}
Let $Y$ be an invertible element in $\widetilde{U_q(\g)}$ such that the map $C_Y: X \rightarrow Y X Y^{-1}$ restricts to an algebra automorphism of $U_q(\g)$. Then $C_Y$ is a coalgebra anti-automorphism if and only if, for every pair of representations $V$ and $W$, the map
\begin{equation}
\begin{aligned}
\sigma^Y_{V,W}: V \otimes W & \rightarrow W \otimes V \\
v \otimes w &\rightarrow \Flip \circ (Y^{-1} \otimes Y^{-1}) \Delta (Y) v \otimes w
\end{aligned}
\end{equation}
is an isomorphism.
\end{Proposition}

Given an element $Y$ satisfying the conditions of Theorem \ref{coanti}, we use the notation $\sigma^Y$ to denote the system of isomorphisms $\{ \sigma^Y_{V,W} \}$.

\subsection{The element $X$} \label{makeX}

We now explicitly describe the element $X$ from Equation (\ref{Rmake_eq}). We will need a way to specify a lowest weight vector of $V_\lambda$, depending linearly on a chosen of highest weight vector. We do this using the action of the braid group on $V_\lambda$. We very briefly review this theory, and refer the reader to, for example, \cite[Chapter 8.1.2]{CP} or \cite{L} for more details.

 \begin{Definition}[{see \cite[5.2.1]{L}}] \label{defti} 
 $T_i$ is the element of  $\widetilde{U_q(\g)}$ that acts on a weight vector $v$  by:
 \begin{equation*}
 T_i(v)= \sum_{\begin{array}{c} a,b,c \geq 0 \\ a-b+c=(\wt(v), \alpha_i) \end{array}} (-1)^b q_i^{ac-b}E_i^{(a)} F_i^{(b)} E_i^{(c)} v.
 \end{equation*}
\end{Definition}

By \cite[Theorem 39.4.3]{L}, these $T_i$ generate an action of the braid group on each $V_\lambda$, and thus a map from the braid group to $\widetilde{U_q(\g)}$. This realization of the braid group is often referred to as the quantum Weyl group. It is related to the classical Weyl group by the fact that, for any weight vector $v \in V$, $\wt(T_i(v))= s_i(\wt(v))$.

\begin{Definition}
Let $w_0= s_{i_1} s_{i_2} \cdots s_{i_N}$ be a reduced expression for the longest word $w_0$ in the Weyl group. Then $T_{w_0}=  T_{i_1} T_{i_2} \cdots T_{i_N}$.
\end{Definition}

It is clear that $T_{w_0}$ is well defined because the elements $T_i$ satisfy the braid relations. Note that $T_{w_0}$ interchanges the highest and lowest weight spaces of $V_\lambda$.

\begin{Definition}
$v_\lambda^\low$ is the element in the lowest weight space of $V_\lambda$ defined by
\begin{equation*}
T_{w_0}(v_\lambda^\low)= v_\lambda,
\end{equation*}
where $v_\lambda$ is the chosen highest weight vector.
\end{Definition}

\begin{Definition}
$J$ is the element of $\widetilde{U_q(\g)}$ defined by, for any weight vector $v \in V_\lambda$, 
\begin{equation*}
J(v) = q^{(\wt(v), \wt(v))/2+(\wt(v), \rho)}v.
\end{equation*}
\end{Definition}

\begin{Definition}  \label{defX}
$X=JT_{w_0}$. 
\end{Definition}

\begin{Lemma} \label{Xprops} (see \cite[Section 5.2]{Rcommutor}) The element $X$ has the following properties:
\begin{enumerate}
\item \label{llow} $X(v_\lambda^\low)= q^{(\lambda, \lambda)/2+(\lambda,\rho)} v_\lambda$.

\item \label{lhigh} $X(v_\lambda)= (-1)^{\langle 2\lambda, \rho^\vee \rangle} q^{(\lambda, \lambda)/2+(\lambda,\rho)} v_\lambda^\low$.

\item \label{X2prop} $X^2$ is central, and acts on $V_\lambda$ as multiplication by the scalar $(-1)^{\langle 2\lambda, \rho^\vee \rangle}q^{(\lambda, \lambda) + 2 (\lambda, \rho)}$.  

\item $C_X$ is given by
\begin{equation}
\begin{cases}
C_X(E_i)=-F_{\theta(i)}, \\
C_X(F_i)=-E_{\theta(i)}, \\
C_X(K_i)=K_{\theta(i)}^{-1}.
\end{cases}
\end{equation}
\end{enumerate}
\end{Lemma}

The following is key:

\begin{Theorem}[{\cite[Theorem 3]{KR:1990}, \cite[Theorem 1]{LS}} \label{sR} (see \cite{Rcommutor} for this exact statement)]
$\sigma^X$ is the standard braiding. Equivalently, the standard $R$-matrix for $U_q(\g)$ can be realized as
\begin{equation}
R=  (X^{-1} \otimes X^{-1}) \Delta(X).
\end{equation}
\end{Theorem}

In fact, we will take Theorem \ref{sR} as the definition of the universal $R$-matrix.

\subsection{Ribbon Hopf algebras}
\begin{Definition} (see \cite[Definition 4.2.8]{CP}) \label{ribbon_def}
A ribbon Hopf algebra $(\Hopf, R, v)$ is a quasi-triangular Hopf algebra $(\Hopf, R)$ equipped with an invertible central element $v$ such that 
\begin{enumerate}
\item \label{up} $v^2=uS(u)$, where $u= \mu(S \otimes id) R_{21}$
\item $S(v)=v$
\item $\varepsilon(v)=1$
\item \label{deltp} $\Delta(v)=(v \otimes v)(R_{21}R_{12})^{-1}.$
\end{enumerate}
\end{Definition}

\begin{Proposition} \label{makeg} (see \cite{CP})
The element $g:= v^{-1}u$ is grouplike, where $u$ is as in Definition \ref{ribbon_def} part (\ref{up}).
\qed
\end{Proposition}

The following four maps are crucial to studying ribbon Hopf algebras:

\begin{Definition} (see \cite[p. 163]{CP}) 
Let $(\Hopf, R, v)$ be a ribbon Hopf algebra, and $(V, \pi)$ a representation of $\Hopf$.   Let $f \in V^*$ and $v \in V$. Let $e_i$ and $e^i$ be dual bases of $V$ and $V^*$ respectively. Then
\begin{enumerate}

\item \label{evthing} $\ev(f \otimes v)= f(v)$. 

\item \label{qtrbit} $\qtr (v\otimes f) =f(gv).$

\item \label{coevthing} $\coev(1)= \sum e_i \otimes e^i$. 

\item \label{coqtrbit} $\coqtr(1) = \sum  e^i \otimes g^{-1} e_i.$
\end{enumerate} 
\end{Definition}

Recall that $V \otimes V^*$ can be identified with $\End(V)$.  Under this identification $\qtr$ is the quantum trace of \cite{RT2}.

We will be working with $U_q(\g)$, which is not a ribbon Hopf algebra according to the above definition, since the elements $v$ and $R$ actually lie in the completions $\widetilde{U_q(\g)}$ and $\widetilde{U_q(\g) \otimes U_q(\g)}$ discussed in Section \ref{completion}. This is known as a topological ribbon Hopf algebra. The theory goes through just as well in the topological case. 

\begin{Definition} \label{defC}
Let $C$ be the element of $\widetilde{U_q(\g)}$ which acts on the representation $V_\lambda$ as multiplication by $q^{-(\lambda, \lambda)-2(\lambda, \rho)}$
\end{Definition}

The following is well known (see for example \cite[Corollary 8.3.16]{CP}).

\begin{Theorem}  \label{normal_ribbon}
$(U_q(\g), R,C)$ is a topological ribbon Hopf algebra. \qed
\end{Theorem}

\subsection{The Temperley-Lieb category} \label{TLdef}
The unoriented Temperley-Lieb category $TL$ has as objects collections of points on the real line up to isotopy, and a morphism from $A$ to $B$ is a formal linear combination of ``planar arc diagrams" each of which is a collection of nonintersecting segments in $\br \times [0,1]$ whose boundary consists of $A \times 0$ and $B \times 1$ modulo isotopy in the plane and the following relation:

\begin{equation}
\setlength{\unitlength}{0.5cm}
\thicklines
\begin{picture}(16,1.9)
\put(0,-1){
\begin{picture}(16,3.5)
\thicklines
\put(4,1.6){\circle{2.8}}
\put(7.5,1.5){$\longrightarrow$}
\put(10,1.5){$-q-q^{-1}.$}
\end{picture}}
\end{picture}
\end{equation}
\vspace{0.1cm}

As discussed in, for example, \cite{GW}, $TL$ is a rigid monoidal category, where composition is given by vertical stacking, tensor product is given by disjoint union, and the dual is given by 180-degree rotation.  Since $**$ (i.e. double dual) acts trivially this is a pivotal category with the trivial pivotal structure.

\section{Half-ribbon Hopf algebras}

This section contains the definition of a half-ribbon Hopf algebra. It also explains the relationship with the topological category $\HRIB(S)$.

\subsection{Definition and basic properties} \label{main1}
\begin{Definition}
A half-ribbon Hopf algebra is a Hopf algebra, along with an element $t \in \Hopf$, such that $(\Hopf, R, v)$ is a ribbon Hopf algebra, where $R= (t^{-1} \otimes t^{-1}) \Delta(t)$ and $v= t^{-2}$.

In the case where $t$ only exists in a completion of $\Hopf$, we say $\Hopf$ is a topological half-ribbon Hopf algebra.
 \end{Definition}
 
 \begin{Proposition} \label{hrprops}
 A half-ribbon element $t$ in a half-ribbon Hopf algebra has the following properties:
 \begin{enumerate}
 \item \label{aproperty} The algebra automorphism $C_t: \Hopf \rightarrow \Hopf$ defined by $x \rightarrow txt^{-1}$ is also a coalgebra anti-automorphism.
 
 \item \label{eproperty} $\varepsilon(t)=1.$ Equivalently, $t$ acts as the identity on the trivial representation.
 
 \item \label{sproperty} $S(t)^2= t^2.$
 \end{enumerate}
 \end{Proposition}
 
 \begin{proof}
 (\ref{aproperty}) follows from $R \Delta(x)= \Delta^{op}(x)R$.
 
 (\ref{eproperty}) follows from the fact that $R$ acts as the identity on ${\bf 1} \otimes V$, where ${\bf 1}$ is the trivial representation. 
 
 (\ref{sproperty}) Follows because $t^{-2}$ is a ribbon element so $S(t^{-2})=t^{-2}$.
\end{proof}

\begin{Proposition} \label{theelements} The various important elements in a ribbon Hopf algebra can be written in terms of a half-ribbon elements $t$ as follows: 
\begin{equation}
\begin{array}{c|c|c|c}
R & v & u & g \\ \hline
(t^{-1} \otimes t^{-1}) \Delta(t)= \Delta^{op}(t) (t^{-1} \otimes t^{-1}) & t^{-2} & S(t^{-1})t^{-1} & S(t)t^{-1} = S(t^{-1}) t\\
\end{array}
\end{equation}
\end{Proposition}

\begin{proof}

(R): The first formula is part of the definition of a half-ribbon element. The second formula follows by Proposition \ref{hrprops} part (\ref{aproperty}) since this implies that, for all $x \in \Hopf$, 
$$(t^{-1} \otimes t^{-1}) \Delta(x) (t \otimes t) = \Delta^{op}(t^{-1}xt).$$ 

(v): This is part of the definition of a half-ribbon element.

(u): By definition, $u = \mu \circ (S \otimes 1) R_{21}$, where $\mu$ is multiplication. Thus:
\begin{align}
u &= \mu \circ (S \otimes 1) ([(\Delta^{op}(t)(t^{-1} \otimes t^{-1})]_{21}) \\
&= \mu [((S \otimes 1)  (\Delta(t))(t^{-1} \otimes t^{-1})] \\
&\label{hj0}= \sum S(t^{-1}) S(t_1) t_2 t^{-1}\\
&\label{hj1} =  S(t^{-1})[\iota \circ \varepsilon (t)] t^{-1} \\
&\label{hj2} =  S(t^{-1})t^{-1}.
\end{align}
In Equation (\ref{hj0}) we have used Sweedler's notation $\Delta(x) = \sum x_1 \otimes x_2$.  Equation (\ref{hj1}) follows from the antipode axiom of a Hopf algebra, and (\ref{hj2}) follows by Proposition \ref{hrprops} part (\ref{eproperty}). 

(g) By definition, $g= v^{-1} u$. Recall that $v$ is central. To get the second equation, apply Lemma \ref{hrprops} part (\ref{sproperty}). 
\end{proof}

\begin{Comment}
For historical reasons the ribbon element $v$ represents a negative twist.  We use $t$ for the positive half-twist, which explains the fact that $v=t^{-2}$.
\end{Comment}

\subsection{A topological category of ribbons with half-twists}
There is a functor from a certain topological category to the category of representations of any ribbon Hopf algebra, which allows one to construct topological invariants. For a half-ribbon Hopf algebra, this functor can be extended to a larger topological category. In this section we define the two relevant categories, beginning with the large one.

\begin{Definition}
 $\THRIB(S)$ (topological half-ribbons with labels in some set $S$) is the category defined by:
 
$\bullet$ Objects in $\THRIB(S)$  consist of a finite number of disjoint closed intervals on the real line, each labeled with an element of $S$, each with a choice of shading (shaded or unshaded), and each directed (up or down). These objects are considered up to isotopy of the real line. For example:
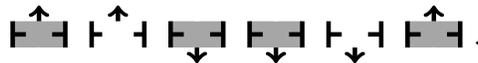
\vspace{-0.4cm}
\begin{center}
$$\suv[A] \uuv[B] \sdv[B] \sdv[C] \udv[A] \suv[D].$$
\end{center}

$\bullet$ A morphism between two objects $A, B \in \THRIB(S)$ consist of a `tangle of orientable, directed ribbons'  in $\br^2 \times I$, whose `loose ends" are exactly $(A, 0, 0)  \cup (B,0, 1) \subset \br \times \br \times I $, along with a choice of direction and shading of each ribbon. These must be such that the direction (up or down) of each interval in $A \cup B$ agrees with the direction of the ribbon whose end lies at that interval, and the shading (light or dark) of each interval agrees with the shading on the visible side of the ribbon near that interval.  
These are considered up to isotopy.

$\bullet$ Composition of two morphisms is given by stacking them on top of each other, and them shrinking the vertical axis by a factor of two. Note that we read our diagrams bottom to top. For example,

\setlength{\unitlength}{0.7cm}
\begin{center}
\begin{picture}(12,3.6)

\put(0,0){\Ltst}
\put(0,-0.7){\psuv[V]}
\put(2,-0.7){\psuv[U]}
\put(0,2.7){\puuv[U]}
\put(2,2.7){\puuv[V]}

\put(3.5,1.5){$\circ$}

\put(4,2.7){\psuv[V]}
\put(6,2.7){\psuv[U]}
\put(4,0){\Dtst}
\put(4,-0.7){\puuv[U]}
\put(6,-0.7){\puuv[V]}

\put(7.5, 1.5){$=$}

\put(10,0){\ShrunkDtst}
\put(10,1.5){\ShrunkLtst}
\put(10,-0.7){\puuv[U]}
\put(12,-0.7){\puuv[V]}
\put(10,2.7){\puuv[U]}
\put(12,2.7){\puuv[V]}

\put(13.5,0){$.$}

\end{picture}
\end{center}

\end{Definition}

\begin{Definition}
$\DR(S)$ is the full subcategory of $\THRIB(S)$ consisting of those objects all of whose intervals are unshaded.
\end{Definition}

The definition of $\THRIB(S)$ can easily be made rigorous, in the same way as is done for the subcategory  $\DR(S)$ (see for example \cite{CP}). If fact, as with $\DR(S)$, $\THRIB(S)$ is  a rigid monoidal category, where tensor products and duals are are shown below:
\begin{equation}
\left(
\suv[A] \uuv[B]
\right) 
\otimes 
\left(\udv[A] \uuv[C] \suv[C]
 \right) 
 =
\suv[A] \uuv[B] \udv[A] \uuv[C] \suv[C],
\end{equation}
\begin{equation}
\left(\suv[A] \udv[B] \sdv[C] \suv[B] \right)^* = \sdv[B] \suv[C] \uuv[B] \sdv[A].
\end{equation}

\subsection{A functor from $\THRIB(S)$ to $\Hopf$-rep}

We now show that the category of representations of a half-ribbon Hopf algebra admits a natural functor from $\THRIB(\Hopf\text{-rep})$. This theorem is an extension of the corresponding result relating ribbon Hopf algebras and $\DR(\Hopf\text{-rep})$, so we begin by stating this known result. 

\begin{Theorem} \label{ribbonfunctor} (see \cite{R}, \cite{Shum}, or \cite[Theorem 5.3.2]{CP}) Let $(\Hopf, R,C)$ be a ribbon Hopf algebra. 
There is a unique monoidal functor $\cF'$ from $\DR(\Hopf \text{-rep})$ to $\Hopf \text{-rep}$ such that
\begin{enumerate}

\item $\cF'(\uuv[V])=V$ and $\cF'(\udv[V])=V^*$,

\item $\begin{array}{ll} \cF' \left( \uducap[V]  \right)= \ev, & \cF' \left( \uudcap[V] \right)= \qtr, \\ \\
\cF' \left( \uudcup[V] \right)= \coev, & \cF' \left( \uducup[V] \right)= \coqtr,
\end{array}$
\vspace{0.3cm}

\item $\cF' \left( \begin{picture}(2,2.2)
\put(0,0.9){\twist}
\put(0,-1.1){\usltwist}
\end{picture} \right) = C,$ 

\noindent thought of as a morphism from $V$ to $V$ or from $V^*$ to $V^*$, depending on the orientation.

\vspace{0.05in}

\item $\cF'\left( \plaincrossing \right) = \Flip \circ R$

\noindent as a morphism from the tensor product of the bottom two objects to the tensor product of the top two objects, regardless of labeling and orientation.

\end{enumerate}
\end{Theorem}

\begin{Comment}
Warning, this functor is not a strict rigid functor.  Notice that $\cF'(\udv[V]{}^*)=\cF'(\uuv[V]) = V$, while $\cF'(\udv[V])^*=(V^*)^* = V^{**}$.
\end{Comment}

Our main result for this section is that $\cF'$ can be extended to a functor $\cF: \THRIB(\Hopf \text{-rep}) \rightarrow \Hopf \text{-rep}$ as follows.

\begin{Definition}
Let $V$ be a representation of a half-ribbon Hopf algebra $\Hopf$ and for each $x \in \Hopf$ let $\pi_V(x)$ be the element of $\End(V)$ defined by $x$. Define a new representation $V^{C_t}$ which is equal to $V$ as a vector space, but with the action of $x \in \Hopf$ defined by $\pi_{V^{C_t}}(x) = \pi_V(txt^{-1})$. \end{Definition}

\begin{Definition}
The `topological half-twist' is the morphism $I_n$ in $\HRIB$ which takes $n$ ribbons of any shading and orientation, and twists them all together by 180 degrees, as shown below:

\begin{center}
\setlength{\unitlength}{0.3cm}
\begin{picture}(10,5.2)
\put(0,0){\ttopns}
\put(3.9,4.8){$. . .$}
\put(3.9,0){$. . .$}
\end{picture}
\setlength{\unitlength}{0.3cm}
\end{center}
\end{Definition}

\begin{Theorem} \label{nice-hfunctor}
There is a unique monoidal functor $\cF$ from $\mathcal{HRIB}(\Hopf \text{-rep})$ to $\Hopf \text{-rep}$ such that
\begin{enumerate}

\item $\cF(\uuv[V])=V$, $\cF(\udv[V])=V^*$, $\cF(\suv[V])= V^{C_t}$ and  $\cF(\sdv[V])= (V^*)^{C_t}.$

\vspace{0.1cm}

\item On unshaded caps and cups $\cF$ agrees with $\cF'$.

\item \label{thtp} The topological half-twist on $n$ ribbons is sent to $\rev \circ \Delta^n(t)$ acting on $\cF$ of the bottom object, regardless of shading and orientation.
\end{enumerate}
Furthermore, the restriction of $\cF$ to $\DR(\Hopf \text{-rep})$ agrees with $\cF'$. 
\end{Theorem}

The proof of Theorem \ref{nice-hfunctor} requires several lemmas, and will occupy the rest of this section. 

\begin{Comment} \label{lr_comment}
As was the case with $\cF'$, $\cF$ is not a strict rigid functor.  In particular, $\cF(\sdv[V])= (V^*)^{C_t}$, \vspace{0.15cm}which is not equal to  $(V^{C_t})^*$ (although these two representations are isomorphic). In fact, $(V^*)^{C_t}$ is equal to the right dual ${}^*(V^{C_t})$. To see why, notice that, as vector spaces, both $(V^*)^{C_t}$ and ${}^*(V^{C_t})$ are equal to $\Hom_F(V, F)$. Each comes with a chosen action of $\Hopf,$ and, using the fact that $C_t$ is a coalgebra antiautomorphism, one can show that these two actions are identical. 
\end{Comment}

\begin{Comment}
A good mnemonic for remembering what $\cF$ does to objects is to think of the ribbon as always being labeled with a representation on its light side, so that when you look at the dark side you see that label rotated 180-degrees around the y-axis.  For example, instead of $V_i^{C_t}$ you would see ${}_i V$.  This works for remembering how duals interact with half twists since when you rotate $V_i^*$ by 180-degrees about the $y$-axis you see ${}_i^*V$.
\end{Comment}

\begin{Comment}
There is more than one extension of $\cF'$ from $\Ribbon(\Hopf\text{-rep})$ to the larger category $\HRIB(\Hopf\text{-rep})$, but only one satisfies Theorem \ref{nice-hfunctor} Part \eqref{thtp}.
\end{Comment}

We first prove the following different characterization of $\cF$. This is simpler to prove, and is also more general than Theorem \ref{nice-hfunctor}. However, this result is less satisfying in other ways. For instance, Proposition \ref{hfunctor} does not imply that all half twists are sent to $t$ acting on $\cF$ of the bottom object. 

\begin{Proposition} \label{hfunctor}
 Let $(\Hopf, R, v)$ be a ribbon Hopf algebra, and let $t \in \Hopf$ be an invertible element.  There is a unique monoidal functor $\cF$ from $\THRIB(\Hopf \text{-rep})$ to $\Hopf \text{-rep}$ such that:
\begin{enumerate}

\item $\cF$ agrees with $\cF'$ on the full subcategory $\DR(\Hopf \text{-rep})$.

\item $\cF(\suv[V])= V^{C_t}$ and $\cF(\sdv[V]) = (V^*)^{C_t}$.
\vspace{0.3cm}

\item \label{hh} $\cF \left( \usrtwist \right) = t$ as a morphism from $V$ to $V^{C_t}$ or from  $V^*$ to $(V^*)^{C_t}$ depending on orientation.

\noindent $\cF \left( \sultwist \right) = t^{-1}$ as a morphism from $V^{C_t}$ to $V$ or from $(V^*)^{C_t}$ to $V^*$  depending on orientation (notice that this is a \emph{negative} half-twist, hence the use of $t^{-1}$). 
\end{enumerate}
\end{Proposition}
\begin{proof}
This is the simplest case of Reshetikhin and Turaev's more general extension of $\cF'$ to ribbon graphs with coupons \cite{RT1}, where we think of the half-twists defined above as coupons on a single ribbon.  So we just sketch the proof.  Every morphism in $\THRIB(\Hopf \text{-rep})$ is isotopic to ribbon in the following standard form: every shaded object at the bottom has a positive half-twist next to it, then in the middle of the diagram there's a ribbon tangle, then at the top of the diagram every shaded object has a negative half-twist next to it.  Furthermore two such diagrams are isotopic if and only if the ribbon tangles in the middle of their standard diagrams are isotopic.  Hence there is at most one possible functor $\cF$.  That $\cF$ is in fact a functor follows from the fact that $t$ and $t^{-1}$ are inverses of each other, so that composition of diagrams is just given by composing their middle parts.
\end{proof}

The following lemma shows how $\cF$ acts on those elementary morphisms not already specified in the statement of Proposition \ref{hfunctor}.

\begin{Lemma} \label{shadedinfo}
If $(\Hopf, t)$ is a half-ribbon Hopf algebra, and $\cF$ is defined as in Proposition \ref{hfunctor}, then the following formulas hold.

\begin{enumerate} 
\item $\cF$ sends any positive half-twist to $t$ acting on $\cF$ of the bottom object regardless of orientation or shading.  $\cF$ sends any negative half-twist to $t^{-1}$ regardless of orientation or shading.  

\item \label{crossingcondition} $\cF$ takes any simple crossing to $\Flip \circ R$ acting on $\cF$ of the bottom object, regardless of orientations and shadings.

\item \label{capscondition} $\cF$ sends shaded caps and cups to the following maps 
\begin{align}
\label{cupcap1} & \cF \left( \sudcap[V] \right) =  \ev_{{}^*(V^{C_t})}: v \otimes f \mapsto f(v) \\
\label{cupcap2} & \cF \left( \sducap[V] \right)=  \qtr_{{}^*(V^{C_t})} :f \otimes v \mapsto f(\pi_V(g) v) \\
\label{cupcap3} & \cF \left( \sudcup[V] \right)=\coqtr_{{}^*(V^{C_t})}: 1 \mapsto \sum_i \pi_V(g)^{-1} v_i \otimes v^i \\
\label{cupcap4} & \cF \left( \sducup[V] \right) \coev_{{}^*(V^{C_t})}: 1\mapsto \sum_i v^i \otimes  v_i.
\end{align}
For the explicit formulas, we have used the fact that $V^{C_t}$ (respectively $(V^*)^{C_t}$) is identical to $V$ (respectively $V^*$) as a vector spaces, and chosen dual basis $v_i$ and $v^i$ for $V$ and $V^*$. 
\end{enumerate}
\end{Lemma}

\begin{Comment}
In order to apply the maps $ \ev_{{}^*(V^{C_t})}$ etc. above, one must use the canonical isomorphism $({}^*(V^{C_t}))^* \cong V^{C_t}$, and the equality $(V^*)^{C_t} = {}^*(V^{C_t})$ from Comment \ref{lr_comment}. It is a straightforward excercise to show that these maps are given in coordinates by the above formulas. 
\end{Comment}

\begin{proof}[Proof of Lemma \ref{shadedinfo}]
We prove each of these formulas by a direct calculation.

For the half-twists and ribbon full-twists we only use that $t^2 = v^{-1}$.  For example, the positive half-twist going from unshaded to shaded is a composition of the negative half-twist with the full twist.  So this result follows from $t^{-1} v^{-1} = t$.

The formulas for $\cF$ applied to a crossing with one or both of the ribbons shaded all follow from the naturality of the braiding. For example, to find $\cF$ of a crossing where one strand is shaded, consider the following isotopy:
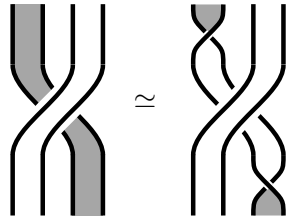
\begin{align*}
\sucrossingisotopy
\end{align*}
Thus $\cF$ of the left side must be equal to $(t \otimes 1) \circ \Flip \circ R  \circ (1 \otimes t^{-1})$. By the naturality of the braiding, 
\begin{align}
(t \otimes 1) \circ \Flip \circ R  \circ (1 \otimes t^{-1}) &= \Flip \circ R \circ (1 \otimes t) \circ (1 \otimes t^{-1}) = \Flip \circ R.
\end{align}
This holds independently of the orientations of the ribbons. A similar argument gives the same formula for crossings where the other ribbon is shaded, or where both are shaded.

It remains to compute the formula for shaded cups and caps.   These equations use $g = S(t) t^{-1}= S(t^{-1}) t$.  We explicitly show two of the four cases. The other two are similar.

By the definition of $\cF$, one has
\begin{align}
\cF \left( \sudcap[V] \right)(v \otimes f) &=
\cF \left( \begin{picture}(4,3)
\put(2,-1.75){\psdv[V]}
\put(0,-1){\ptwistsub}
\put(0,.9){\ucap}
\put(2,-1){\pnegtwistsub}
\put(0,-1.75){\psuv[V]}
\end{picture}
 \right) (v \otimes f) \label{shadedcap1} \\
&= \qtr (t v \otimes t^{-1} f) \label{shadedcap2} \\
& = t^{-1} f(g t v) \label{shadedcap3} \\
&= f(S(t^{-1}) g t v)\label{shadedcap4} \\
& = f(v) \label{shadedcap5}.
 \end{align}
Here (\ref{shadedcap1}) is an isotopy, (\ref{shadedcap2}) follows from functoriality and our computation of $\cF$ on all half-twists, (\ref{shadedcap3}) is the definition of $\qtr$, (\ref{shadedcap4}) is the definition of the action on the dual space, and (\ref{shadedcap5}) uses the formula $g =  S(t) t^{-1}$.

Now for a cup. Notice that, if $e^i \in V^*$ and $e_j \in V$ are dual bases and $x$ is an invertible element in $\Hopf$, then $v^i = x e^i$ and $v_i = S(x)^{-1}e_i$ are also dual bases. By the definition of $\cF$, we have:
\begin{align}
\cF \left( \sudcup[V] \right)(1) &=
\cF \left( \begin{picture}(4,3)
\put(0,1.5){\psuv[V]}
\put(0,-0.2){\ptwistusb}
\put(2,-0.2){\pnegtwistusb}
\put(2,1.5){\psdv[V]}
\put(0,-1.6){\ucup}
\end{picture}
 \right) (1) \label{shadedcup1} \\
&=  \sum t e_i \otimes t^{-1} e^i \label{shadedcup2} \\
&= \sum t S(t^{-1}) v_i \otimes v^i \label{shadedcup3} \\
&= \sum g^{-1} v_i \otimes v^i \label{shadedcup4}
\end{align}
Here (\ref{shadedcup1}) is an isotopy, (\ref{shadedcup2}) follows from functoriality and our computation of $\cF$ on all half-twists, (\ref{shadedcup3}) is the change of dual bases described above with $x = t$, and (\ref{shadedcup3}) uses the formula $g = S(t) t^{-1}$.
\end{proof}

Note that above we used the formulas relating $t$ to $g$ (which was derived using the formula relating $R$ and $t$), but we did not use the formula for $R$ directly.  That formula is used in the following proposition which computes how $\cF$ acts on the half-twist applied to many strands.

\begin{Proposition} \label{bighalftwist} Suppose that $(\Hopf,t)$ is a half-ribbon Hopf algebra, let $\cF$ be the unique functor guaranteed by Proposition \ref{hfunctor}.  Then, $\cF(I_n)=  \rev \circ \Delta^n(t).$ 
\end{Proposition}

\begin{proof} We proceed by induction on $n$, the case $n=1$ being trivial. 
Consider the isotopy
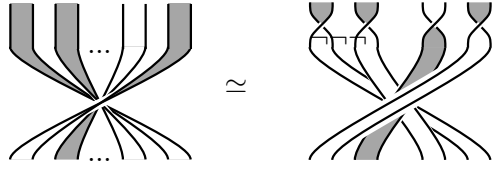
\begin{center}
\setlength{\unitlength}{0.3cm}
\begin{picture}(20, 7.3)
\put(0,0){\ttopns} 
\put(-0.417,4.9){\ssid}
\put(1.583,4.9){\ssid}
\put(4.583,4.9){\suid}
\put(6.583,4.9){\ssid}
\put(3.9,4.8){$. . .$}
\put(3.9,0){$. . .$}

\put(10,3){$\simeq$}
\put(12,-0.5){\makeit}
\put(22, 0){$.$}
\end{picture}
\setlength{\unitlength}{0.7cm}
\end{center}
Thus, by the definition of $\cF$, 
\begin{equation} \label{FI}  \cF(I_n)= (t \otimes t \otimes \cdots \otimes t) \cF(T^{(n)}_{w_0}), \end{equation}
where $T^{(n)}_{w_0}$ is the braid group element corresponding to the longest element of $S_n$. 
Let $\sigma^{(n)}$ be the braid group element 

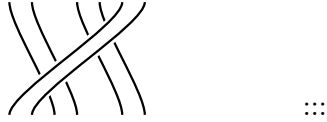
\begin{center}
\dasigma
\end{center}
Let $s^{(n)}$ be the image of $\sigma^{(n)}$ in $S_n$, and let $\rev^{(n)}$ denote the longest element of $S_n$. Clearly $T_{w_0}^{(n)} = (T_{w_0}^{(n-1)} \otimes 1) \circ \sigma^{(n)}.$ Hence
\begin{align}
\cF (I_n) 
\label{r2}&= (t \otimes t \otimes \cdots \otimes t) \circ \cF(T_{w_0}^{(n)}) \\
\label{r3}&= (t \otimes t \otimes \cdots \otimes t) \circ (\cF(T_{w_0}^{(n-1)}) \otimes 1) \circ \cF(\sigma^{(n)}) \\
\label{r4}&= (\cF(I_{n-1}) \otimes t) \circ \cF(\sigma^{(n)}) \\
\label{r5}&=  \left(\rev^{(n-1)} \Delta^{n-1} (t) \otimes t \right) \circ s^{(n)} \circ (1 \otimes \Delta^{n-1}) (R) \\
\label{r6}&=  \left( \rev^{(n-1)} \Delta^{n-1} (t) \otimes t \right) \circ s^{(n)} \circ (1 \otimes \Delta^{n-1}) \left( (t^{-1} \otimes t^{-1})\Delta (t) \right) \\
\label{r7} &= \rev^{(n)} \Delta^n(t).
\end{align}
Here \eqref{r4} follows from Equation \eqref{FI} in the case $n-1$, \eqref{r5} follows from the inductive assumption and the quasitriangularity of $R$, \eqref{r6} holds because $t$ is a half-ribbon element, and \eqref{r7} is a straightforward calculation.

Note that this proof works regardless of the orientations and shadings of the ribbons.
\end{proof}

\begin{proof}[Proof of Theorem \ref{nice-hfunctor}]
Using the isotopy from Equation \eqref{main_isotopy}, the crossing is a composition of the two strand topological half-twist and two copies of the inverse of the one strand topological half-twist. Thus the conditions in the statement uniquely determine $\cF$ on all the elementary morphisms listed in the condition of Proposition \ref{hfunctor}. So Proposition \ref{hfunctor} shows that there is a unique candidate for $\cF$. Proposition \ref{bighalftwist} shows that this $\cF$ satisfies the remaining condition of the theorem.
\end{proof}

\section{Ribbon and Half-ribbon elements for $U_q(\g)$} \label{fiveaa}

We show that $U_q(\g)$ is always a topological half-ribbon Hopf algebra. That is, there exists an element $X$ in a completion of $U_q(\g)$ such that $(X^{-1} \otimes X^{-1}) \Delta(X)$ is the standard $R$-matrix, and $X^{-2}$ is a ribbon element. Interestingly, $X^{-2}$ is not the standard ribbon element $C$. We also classify the different ribbon elements for $U_q(\g)$, and discuss how one might decide which of these arise from half-ribbon elements. In particular, we show that the standard ribbon element for $U_q(\fsl_2)$ does not arise from a half-ribbon element. 
We then explain some consequences of varying the ribbon element. It turns out the the ribbon element $X^{-2}$ is in some ways particularly nice (see Lemma \ref{fs1} and Section \ref{dc_section}). 

\subsection{A half-ribbon Hopf algebra structure on $U_q(\g)$}

We first need the following result relating different ribbon elements for the same quasi-triangular Hopf algebra.

\begin{Lemma} \label{other_ribbons}
Let $(\Hopf, R,v)$ be a ribbon Hopf algebra, and $s \in \Hopf$ be a central grouplike element that squares to 1. Then  $(\Hopf, R, v s)$ is also a ribbon Hopf algebra. Furthermore all ribbon elements for $(\Hopf, R)$ are of the form $vs$ for some such $s$.
\end{Lemma}

\begin{proof}
This is a straightforward application of the definition of a ribbon element (cf. \cite[Remark 3.4]{Spherical}).
\end{proof}

\begin{Theorem} \label{ishalf} Let $X$ be the element from Theorem \ref{sR}. Then 
$(U_q(\g), X)$ is a topological half-ribbon Hopf algebra. Furthermore $(X^{-1} \otimes X^{-1}) \Delta(X)$ is the standard $R$-matrix.
\end{Theorem}

\begin{proof}
By Theorem \ref{sR}, $(U_q(\g), R)$ is a quasi-triangular Hopf algebra, where $R= (X^{-1} \otimes X^{-1}) \Delta(X)$ is the standard $R$ matrix. Thus it suffices to show that $X^{-2}$ is a ribbon element for this quasi-triangular Hopf algebra. We already know that $(U_q(\g), R,C)$ is a ribbon Hopf algebra. Thus by Lemma \ref{other_ribbons} it suffices to show that 
\begin{enumerate}
\item $X^2C$ is central 
\item $X^2C$ is grouplike
\item $(X^2C)^2=1$.
\end{enumerate}

(i): By Lemma \ref{Xprops} part (\ref{X2prop}) and Definition  \ref{defC}, $X^2C$ acts on $V_\lambda$ as multiplication by the scalar $(-1)^{\langle 2\lambda, \rho^\vee \rangle}$. Thus it is clearly central. 

(ii): Any highest weight $\nu$ for $V_\lambda \otimes V_\mu$ has weight $\lambda + \mu - \gamma$ for some $\gamma$ in the root lattice. Since $\langle 2\gamma, \rho^\vee \rangle$ is an even integer for any $\gamma$ in the root lattice, it follows that $X^2C$ is grouplike.

(iii): $(X^2 C)^2$ acts on $V_\lambda$ by  $(-1)^{2 \cdot \langle 2\lambda, \rho^\vee \rangle} = 1$.
\end{proof}

\subsection{The Frobenius-Schur indicator} We must now discuss a tool for comparing ribbon elements. 
The following definition of the Frobenius-Schur indicator for a pivotal category was given by \cite{PivotalFS} (see also, \cite{HopfFS} for the case of Hopf algebras and \cite{FS} for $C^*$ sovereign categories). 

Recall that a pivotal structure is a natural collection of maps $p_V: V \rightarrow V^{**}$ which defines a natural isomorphism $\mathrm{Id} \rightarrow **$ of monoidal functors \cite{Spherical}. The category of representations of a ribbon Hopf algebra is always pivotal with the pivotal structure being given by the grouplike element $g = v^{-1}u$ as follows: Fix $v \in V$. Then $p_V(v)$ is the element of $V^{**}$ defined by, for all $f \in V^*,$ $p_V(v)(f)= f(gv)$ (see for example \cite{CP}).

\begin{Definition}
Given an $F$-linear category with a chosen pivotal structure $p$, define the Frobenius-Schur indicator of an absolutely simple object $V$ as follows:  If $V \not \cong V^*$, then $FS_p(V) = 0$.  Otherwise, choose an isomorphism $f: V \rightarrow V^*$.  By Schur's lemma there exists some constant, which we define to be $FS_p(V)$, such that $f = FS_p(V)  f^* \circ p$. 
\end{Definition}

\begin{Comment}
As shown in \cite{PivotalFS},  $FS_p(V)$ does not depend on the choice of $f$, and we have $FS_p(V) = \pm 1$ (or 0).
\end{Comment}

\begin{Comment}
Notice that the Frobenius-Schur indicator depends on the pivotal structure, which in turn depends on the choice of ribbon element.  As described in \cite{PivotalFS}, we use $FS_v(V)$ to denote the Frobenius-Schur indicator for $V$, calculated with the ribbon element $v$. 
\end{Comment}

\begin{Comment}
The reason for the term Frobenius-Schur indicator is that in the case of groups (and indeed involutory Hopf algebras \cite{HopfFS}) with the trivial pivotal structure, it agrees with the usual Frobenius-Schur indicator, which is $1$ when there is an invariant orthogonal form on $V\otimes V$, $-1$ when there is an invariant symplectic form on $V\otimes V$, and $0$ otherwise.
\end{Comment}

\begin{Lemma} \label{fs1} Let $V_\lambda$ be an irreducible representation of $U_q(\g)$ and assume $V_\lambda$ is self-dual. Then $FS_{X^{-2}}(V_\lambda) = 1$.
\end{Lemma}
\begin{proof} 
Let $f: V_\lambda \rightarrow V_\lambda^*$ be an isomorphism.
It suffices to show that $f=f^* \circ p_{V_\lambda}$. Since both $f$ and $f^* \circ p_{V_\lambda}$ are isomorphisms, it suffices to show that their actions on $v_\lambda^\low$ agree. But $f(v_\lambda^\low)$ is in the lowest weight space of $V_\lambda^*$, which is one dimensional and pairs non-degenerately with $v_\lambda$. Thus, it in fact suffices to show that 
\begin{equation} \label{atos}
f(v_\lambda^\low)(v_\lambda)=  f^* \circ p_{V_\lambda} (v_\lambda^\low)(v_\lambda).
\end{equation}

Using the formula for $g$ given in Proposition \ref{theelements},
\begin{equation} \label{e47}
S(X)= Xg^{-1} = gX.
\end{equation}
Let $
k= (-1)^{\langle 2\lambda, \rho^\vee \rangle} q^{(\lambda, \lambda)/2+(\lambda,\rho)}.$ By Lemma \ref{Xprops} part (\ref{lhigh}), $X (v_\lambda)= k (v_\lambda^\low).$ Thus the left side of \eqref{atos} can be simplified as follows
\begin{align}
\label{lt1} f(v_\lambda^\low)(v_\lambda) &= f(k^{-1} X(v_\lambda))(v_\lambda) \\
\label{lt2} &= k^{-1} (X f(v_\lambda))(v_\lambda) \\
\label{lt3} &= k^{-1} f(v_\lambda)(S(X) v_\lambda) \\
\label{lt4} &= k^{-1} f(v_\lambda)(Xg^{-1} v_\lambda) \\
\label{lt42} &= k^{-1} f(v_\lambda)(gX v_\lambda) \\
\label{lt5} &= k^{-1} f(v_\lambda) (k g v_\lambda^\low) \\
\label{lt6} &= f(v_\lambda)(g v_\lambda^\low).
\end{align}
Here (\ref{lt2}) follows because $f$ is an isomorphism, (\ref{lt4}) and (\ref{lt42}) follow by Equation (\ref{e47}) and \eqref{lt5} by the above formula for the action of $X$ on $v_\lambda$.

Now consider the right side of \eqref{atos}. By definition, $p_{V_\lambda}(v_\lambda^\low)$ is the element of $V_\lambda^{**}$ which takes $\phi \in V_\lambda^*$ to $\phi(g v_\lambda^\low).$ Thus
\begin{equation} 
f^* \circ p_{V_\lambda} (v_\lambda^\low)(v_\lambda)= f(v_\lambda)(g v_\lambda^\low).
\end{equation}
We have shown that the two sides of \eqref{atos} agree, so the result follows.
 \end{proof}

In \cite{Turaev}, Turaev calls the property that Lemma \ref{fs1} holds unimodality and requires it when defining geometric 3j and 6j symbols. 

\begin{Lemma} \label{likenormal}
$FS_C(V)$ is the usual Frobenius-Schur indicator of the corresponding classical representation.
\end{Lemma}
\begin{proof}
Since the Frobenius-Schur indicator is a discrete invariant, it doesn't change under continuous deformation.
\end{proof}

\subsection{Which ribbon elements for $U_q(\g)$ arise from half-ribbon elements?} \label{no_half}
Here we classify ribbon elements for $U_q(\g)$. We then prove that the standard ribbon element for $U_q(\fsl_2)$ is not equal to $t^{-2}$ for any half-ribbon element $t$.

\begin{Definition}
Suppose that $\phi$ is a character of $P/Q$, where $P$ is the weight lattice of $\g$, and $Q$ is the root lattice.  Define $s(\phi) \in \widetilde{U_q(\g)}$ to act on $V_\lambda$ as multiplication by the scalar $\phi(\lambda)$. 
\end{Definition}

\begin{Theorem} \label{centrals}
The central grouplike elements of $\widetilde{U_q(\g)}$ are precisely the $s(\phi)$.
\end{Theorem}
\begin{proof}
Suppose that $s$ is a central grouplike element of $\widetilde{U_q(\g)}$.  Since it is central, $s$ acts by a scalar on any irreducible representation.  Furthermore, since it is grouplike, $s$ acts by the same scalar on any two irreducible representations that appear in the same arbitrary tensor product of irreducible representations.  Thus such an $s$ gives a function of $P/\sim$, where $\sim$ is the equivalence relation on weights generated by $\lambda \sim \mu$ if $V_\lambda$ and $V_\mu$ appear in the same tensor product of irreducible representations. It suffices to show that $\lambda \sim \mu$ if and only if $\lambda+Q=\mu+Q \in P/Q$. 
One can easily see that $\sim$ respects the additive structure of the weight lattice, so it is enough to prove that $V_\lambda$ appears in some tensor product also containing the trivial representation $V_0$ precisely when $\lambda$ is in the root lattice.

Since the Littlewood-Richardson coefficients are the same in the quantum and classical cases, this reduces to the same question in the classical case.  So let $G$ be the simply connected Lie group attached to $U(\g)$, and let $Z(G)$ be the center of $G$.  The central character map gives a pairing between $Z(G)$ and the weight lattice mod the root lattice.  Since any tensor product of irreducible representations has a well-defined central character, we see that $V_\lambda$ is equivalent to $V_0$ only if $\lambda$ is in the root lattice.

To see the other direction, let $V$ be a faithful representation of the compact adjoint form $K$ of $G$.

{\bf Claim:} For any $\lambda$ in the root lattice, there exists $N$ so that, for all $n > N$, $V_\lambda$ occurs as a subrepresentation of $V^{\otimes n}$.

Proof of claim: Since $\lambda$ is in the root lattice, $V_\lambda$ descends to a representation of $K$.  On this compact form we can use character theory.  Because $K$ has no center, $|\chi_V(g)| < \dim V$ for any $g \neq 1$.  Let 
\begin{equation}
k_n = \left|  \int_K \chi_V^n d\mu \right|,
\end{equation}
where $d\mu$ is the normalized Haar measure. Since $K$ is compact and $|\chi_V(g)|<\chi_V(1)$ for all $g \neq 1$, we see that
\begin{equation}
\lim_{n \rightarrow \infty}  \frac{\langle \chi_{V^{\otimes n}}, \chi_{V_\lambda} \rangle}{k_n} = \lim_{n\rightarrow \infty} \frac{1}{k_n} \int_K \chi_V(g)^n \chi_{V_\lambda}(g^{-1}) d\mu = \chi_{V_\lambda}(1) = \dim V_\lambda.
\end{equation} 
In particular, for $n \geq N$ we see that $\langle \chi_{V^{\otimes n}}, \chi_{V_\lambda} \rangle$ is nonzero, so $V_\lambda$ occurs in $V^{\otimes n}$.  

Applying the above fact to $0$ and an arbitrary root vector $\lambda$, we see that for a sufficiently high power $N' = \text{max}(N_0, N_\lambda)$ the tensor power $V^{\otimes N}$ contains both $V_0$ and $V_\lambda$, and hence $0 \sim \lambda$.
\end{proof}

\begin{Theorem} \label{all_ribbons}
Ribbon elements of $\widetilde{U_q(\g)}$ are exactly elements of the form $s(\phi)X^{-2}$ where $\phi$ is a character of the weight lattice mod the root lattice of order $\leq 2$.
\end{Theorem}
\begin{proof}
This follows from Lemma \ref{other_ribbons} and Theorem \ref{centrals}.
\end{proof}

We have the following relationship between ribbon elements and the Frobenius-Schur indicator.

\begin{Proposition} \label{howFSchanges}
If $V_\lambda$ is self dual, then $FS_{s(\phi) X^{-2}}(V_\lambda) = \phi(\lambda)$.  
\end{Proposition}

\begin{proof}
The only part of the definition of the Frobenious Schur-indicator that changes when you change the ribbon element from $X^{-2}$ to $s(\phi) X^{-2}$ is $p_{V_\lambda}$, and this is multiplied by $ \phi(\lambda)$. Hence the proposition follows from theorem \ref{fs1}. 
\end{proof}

We wish to understand which of these ribbon element extend to a  half-ribbon elements on $\widetilde{U_q(\g)}$.  As we shall see, the ratio of two half-ribbon elements is grouplike, so to classify all of them we need to understand the grouplike elements in $\widetilde{U_q(\g)}$. In general this seems to pose some technical challenges, but we can do the case of $U_q(\fsl_2)$.

\begin{Proposition} \label{twots}
Let $(\Hopf, R)$ be a quasi-triangular Hopf algebra, and assume that $t_1$ and $t_2$ are two half-ribbon elements. Then $t_1 t_2^{-1}$ is grouplike. (Note that we do not assume here that $t_1^{-2}=t_2^{-2}$). 
\end{Proposition}

\begin{proof}
Follows because $(t_1^{-1} \otimes t_1^{-1})  \Delta(t_1)= (t_2^{-1} \otimes t_2^{-1})  \Delta(t_2)$.
\end{proof}

\begin{Lemma} \label{findzero}
Let $V$ be the standard representation of $U_q(\fsl_2)$, with basis $\{ v_+, v_- \}$ such that $E(v_-)= v_+$, $Kv_+=qv_+$, and $Kv_-=q^{-1}v_-$. Then $V \otimes V$ contains a copy of the trivial representation $V_0$. Furthermore, this is spanned by the element
$$v_+ \otimes v_- -q^{-1}v_-\otimes v_+. $$
\end{Lemma}

\begin{proof}
It suffices to check that $\Delta(E)(v_+ \otimes v_- -q^{-1}v_-\otimes v_+)=0.$ This follows from the definition of $\Delta$. 
\end{proof}

\begin{Definition}
Let $a \in \bc(q)$. Define $K_a$ to the the element of $\widetilde{U_q(\fsl_2)}$ which acts on any weight vector $v \in V_\lambda$ by $K_a(v)= a^{\wt(v)}$. Note that $K_a$ is grouplike. 
\end{Definition}

\begin{Lemma} \label{grouplikea}
All grouplike elements in $\widetilde{U_q(\mathfrak{sl}_2})$ are of the form $K_a$ for some $a \in \bc(q)$.
\end{Lemma}

\begin{proof}
A grouplike element is determined by its action on the fundamental representation $V$.  Suppose that $\sigma$ is grouplike. Define constants $a,b,c$ and $d$ by $\sigma v_+ = av_+ +bv_-$ and $\sigma v_-= cv_+ + dv_-$.  Since $\sigma$ is grouplike it must act trivially on the trivial subrepresentation of $V\otimes V$.  Thus by Lemma \ref{findzero},
\begin{align}
v_+ \otimes v_- -q^{-1}v_-\otimes v_+ &= \sigma\otimes \sigma (v_+ \otimes v_- -q^{-1}v_-\otimes v_+) 
\\ &= ac(1-q^{-1})v_+ \otimes v_+ + (ad-q^{-1}bc)v_+ \otimes v_-
\\& -q^{-1}(ad-qbc)v_-\otimes v_+ + (1-q^{-1})bd v_- \otimes v_-.
\end{align}
Comparing coefficients we see that $ac=bd=bc=0$ and $ad=1$.  Hence $\sigma$ acts on $V$ in exactly the same way that $K_a$ does. Since they are both grouplike and $V$ is a tensor generator for the category of representations of $U_q(\fsl_2)$, we see that $\sigma = K_a$.
\end{proof}

\begin{Theorem} \label{nonono}
There is no topological half-ribbon element $t$ for $U_q(\fsl_2)$ such that $t^{-2}$ is the usual ribbon element $C$.
\end{Theorem}
\begin{proof}
By Theorem \ref{normal_ribbon}, $X$ is a half-ribbon element for $(U_q(\fsl_2), R)$. By Lemma \ref{twots}, any other half-ribbon element is of the form $X \alpha$ for some grouplike element $\alpha$. By Lemma \ref{grouplikea}, $\alpha=K_a$ for some $a \in \bc(q)$. Since $X$ sends the $\lambda$ weight space to the $-\lambda$ weight , it follows that $K_a X K_a v = X v$ for any weight space $v$.  Therefore, we have that $(X K_a)^{-2} = X^{-2}$. Hence every half-ribbon element $t$ for $U_q(\fsl_2)$ has $t^{-2}= X^{-2}$. But $X^{-2}$ acts as $-C$ on the standard representation, so is not equal to $C$.
\end{proof}

\subsection{Summary of ribbon and half-ribbon elements for $U_q(\fsl_2)$, $U_q(\fsl_3)$ and $U_q(\fsl_4)$}

We now describe the possible ribbon elements of these three quantum groups, and describe which are the squares of half-ribbon elements. Notice that the three cases all behave differently.

$\bullet$ $U_q(\fsl_2)$. There are two ribbon elements $C$, and $X^{-2}$. Only $X^{-2}$ can be realized using a half-ribbon element (see Theorem \ref{nonono}).

$\bullet$ $U_q(\mathfrak{sl}_3)$. In this case $P/Q \cong \mathbb{Z}/3$.  Thus, there are no central grouplikes of order two, so there can be only one ribbon element. This unique ribbon element can be realized using the half-ribbon element $X$.

$\bullet$ $U_q(\mathfrak{sl}_4)$. In this case $P/Q \cong \mathbb{Z}/4$. This has two characters of order $\leq 2,$ so there are two ribbon elements. In this case both come from half-ribbon elements. A straightforward calculation on the standard representation of $U_q(\fsl_4)$ shows that $X^{-2}$ is not the standard ribbon element. To see the other half-ribbon element, consider the central element $s \in \widetilde{U_q(\fsl_4)}$ which acts on $V_\lambda$ as multiplication by $i^{4(\lambda, \lambda)}$, where $i$ is the complex  fourth root of unity. Recall that $(\lambda, \alpha) \in \bz$ for any root $\alpha$, which implies that $s$ is grouplike. Clearly we also have $s^4=1$, and from these two facts it follows that $X':=sX$ is a half-ribbon element for $(U_q(\fsl_4), R)$. One can check that $s$ acts as mulriplicatin by $-i$ on the standard representation, so $X'^{-2} = s^2 X^{-2} \neq X^{-2}$. Since $X^{-2}$ is the non-standard ribbon element, $X'^{-2}$ must be the standard ribbon element. 

\subsection{Varying the ribbon element and diagrammatic categories} \label{dc_section}

Often times you'll find a result in the literature saying that a certain diagram category ``is the same as" as certain category coming from quantum groups (for example, \cite[Appendix H]{Ohtsuki} \cite[p. 11]{Kuperberg}).  However, in the details of this claim there's an annoying sign difference between the diagram category and the quantum group category.  The reason for this is that, although the two categories match up as braided tensor categories, they are different as pivotal categories. This can be fixed by changing the ribbon element on the quantum group side, thus changing the pivotal structure. 

We illustrate this here by considering the case of $U_q(\mathfrak{sl}_2)$-rep and the Temperley-Lieb category (see Section \ref{TLdef}). In the correspondence discussed in \cite[Appendix H]{Ohtsuki}, the standard representation of $U_q(\fsl_2)$ corresponds to the elementary object (i.e. a single $\bullet$) in the Temperley-Lieb  category. We now use the Frobenius-Schur indicator to show that, for any such statement to hold on the level of pivotal categories, one must use the ribbon element $X^{-2}$

\begin{Proposition}
Let $\bullet$ denote a single point in Temperley-Lieb (with the trivial pivotal structure discussed in Section \ref{TLdef}). Then $FS(\bullet) = 1$.
\end{Proposition}
\begin{proof}
The single strand is an isomorphism between $\bullet$ and $\bullet^*=\bullet$.  The dual of the single strand (given by rotating 180 degrees about the z-axis) is again the single strand.
\end{proof}

\begin{Proposition}
Let $V$ be the standard representation of $U_q(\fsl_2)$. Then
\begin{enumerate}
\item $\text{FS}_C(V)=-1$.
\item $\text{FS}_{X^{-2}}(V)=1$.
\end{enumerate}
\end{Proposition}

\begin{proof}
This is a straightforward calculation using the definitions of $C$ (Definition \ref{defC}) and Lemma \ref{Xprops} part (\ref{X2prop}). It also follows from the discussion in Section \ref{no_half}.
\end{proof}

If $\bullet$ in Temperley-Lieb is going to correspond to the standard representation, the Frobenius-Schur indicators should agree. Clearly this can only happen if we use the ribbon element $X^{-2}$. 

\begin{Comment}
A similar difficulty arises in relating the type $C$ quantum group knot invariants with (a specialization of) the Kauffman polynomial. Once again, one can fix the problem by switching to the ribbon element $X^{-2}$. \end{Comment}

\begin{Comment} This difficulty can also be addressed by changing the diagrammatic category.  The Temperley-Lieb category has another pivotal structure $p$ such that $p_\bullet = -\id_\bullet$.  With this pivotal structure $FS(\bullet)=-1$.  This other pivotal structure is given diagrammatically by the disoriented Temperley-Lieb category of \cite{CMW}.
\end{Comment}

\subsection{The effect of varying the ribbon element on knot invariants}
 
Suppose $(\Hopf, R,v)$ is a ribbon Hopf algebra over a base field $F$.  Fix a simple $\Hopf$-module $V$. The functor from Theorem \ref{ribbonfunctor} sends a link with every component labeled with $V$ to an element of $F$, denoted by $\cF^V(L)$, which is an invariant of framed oriented links. Note that $\cF^V(L)$ actually depends on the choice of ribbon element $v$, so when we need to be clear about which ribbon element we are using we will denote it by $\cF^V_{v}(L)$. Since $V$ is simple the ribbon element acts by a scalar $\theta_v(V)$ on $V$. It follows as in, for example, \cite[\S 3.3]{Ohtsuki} that $\theta_v(V)^{w(L)} \cF_v^V(L)$ is an invariant of oriented but unframed links (where $w(L)$ is the writhe of the link diagram, as described in, for example, \cite[p. 11]{Ohtsuki} \cite{Wikipedia}).

\begin{Proposition} \label{invariantnewribbon}
Suppose that $\Hopf$ is a quasi-triangular Hopf algebra with $\End(\mathbf{1}) = F$ and that $V$ is a simple $\Hopf$-module.  Suppose that $v_1$ and $v_2$ are two different ribbon elements for $\Hopf$.  Let $\theta_{v_1}(V)$ and $\theta_{v_2}(V)$ be the scalars by which $v_1$ and $v_2$ act on $V$.  Let $\cF_{v_1}^V$ and $\cF^V_{v_2}$ be the functors attached to these two ribbon Hopf algebras.
Then for any link $L$, 
\begin{equation*}
\theta_{v_1}^{w(L)}\cF^V_{v_1}(L) = \left( \frac{FS_{v_1}(V)}{FS_{v_2}(V)} \right)^{\#L} \theta_{v_2}^{w(L)}\cF^V_{v_2}(L),
\end{equation*}
where $\#L$ is the number of components of $L$.
\end{Proposition}
\begin{proof}
By theorem \ref{all_ribbons}, $\alpha:= v_1 / v_2$ is central grouplike element of order 2.   Let $\alpha_V$ be the scalar by which $\alpha$ acts on $V$. By Proposition \ref{howFSchanges}, $\alpha_V = FS_{v_1}(V)/FS_{v_2}(V)$.  Let $g_i = v_i^{-1} u$ where $u$ is the Drinfeld element. It follows immediately from definitions that $\qtr_2 = \alpha_V^{-1} \qtr_1$ and $\coqtr_2 = \alpha_V^{-1} \coqtr_1$.

The only elementary morphisms for which $\cF^V_{v_1}$ and $\cF^V_{v_2}$ disagree are left-going cups (corresponding to $\coqtr$), left-going caps (corresponding to $\qtr$), and full-twists.  Thus, 
\begin{equation}
\theta_{v_1}^{w(L)}\cF^V_{v_1}(L) = (\alpha_V)^{-N_L} \theta_{v_2}^{w(L)}\cF^V_{v_2}(L), 
\end{equation}
where 
\begin{equation} 
N_L:= \# \{ \text{left going caps and cups} \} + w(L).
\end{equation}

It is easy to see that $N_L \pmod{2}$ is an invariant of oriented links (for example, check all the Turaev moves \cite[\S 3.2]{Ohtsuki}).  Furthermore, $N_L \pmod{2}$ doesn't change when you replace a positive crossing by a negative crossing.  Since every link can be unknotted by replacing positive crossings with negative crossings, we see that $N_L \pmod{2}$ depends only on the number of components of $L$.  By looking at the $k$-component unlink we see that $N_L \equiv \#L \pmod{2}$.
\end{proof}

\section{Questions}

\begin{Question}
Which ribbon Hopf algebras can be endowed with a half-ribbon element?
\end{Question}

There are several aspects to this question. One could look for examples of half-ribbon Hopf algebras other than $U_q(\g)$. One could also try to find non-examples, in the sense of finding ribbon Hopf algebras that do not contain the required element $t$. Such examples exist (at least for topological ribbon Hopf algebras). For instance we showed in Section \ref{no_half} that $U_q(\mathfrak{sl}_2)$, with the standard ribbon element, cannot be made into a half-ribbon Hopf algebra. However, it can be modified by multiplying the ribbon element by a central grouplike element, and then it does have the required $t$. One could also ask if there are examples of ribbon Hopf algebras that cannot be made into half-ribbon Hopf algebras, even allowing this sort of modification. More ambitiously, one could look for a general method of determining when a ribbon Hopf algebra $\Hopf$ contains an element $t$ such that $(\Hopf, t)$ is a half-twist Hopf algebra.

In the current work we have mainly considered $U_q(\g)$, which is infinite dimensional, and only has a topological half-ribbon element in the sense that $t$ only belongs to a completion of $U_q(\g)$. We feel it would be interesting to look at the case of finite dimensional Hopf algebras as well.

\begin{Question}
Fix a Hopf algebra $\Hopf$. Is there a natural set of conditions one can impose on an element $t \in \Hopf$ which guarantees that $(\Hopf, t)$ is a half-ribbon Hopf algebra?
\end{Question}

We would like to be able to start with a Hopf algebra, which is not a-priori quasi-triangular, and endow it with a ribbon (and half-ribbon) structure by finding a certain $t \in \Hopf$. One can of course write down the conditions $t$ needs to satisfy by insisting that $(t^{-1} \otimes t^{-1}) \Delta(t)$ is a quasi-triangular structure, and $t^{-2}$ is a ribbon element. However, these are very difficult to deal with, so the real question is to find nicer conditions on $t$. This would give a new method of constructing quasi-triangular Hopf algebras. 

\begin{Question}
What happens if you weaken the conditions on $t$?
\end{Question}

Checking that $(t^{-1}\otimes t^{-1})\Delta{t}$ is a quasi-triangular structure is difficult, but for some applications one can weaken this condition by insisting only that
\begin{enumerate}
\item For any representation $V$ and $W$ of $\Hopf$, the map \begin{equation} \Flip \circ (t^{-1} \otimes t^{-1} ) \Delta(t): V \otimes W \rightarrow W \otimes V \end{equation} is an isomorphism.

\item $t^{-2}$ acts as a ribbon element, in the sense that 
\begin{equation}
t^{-2} \circ mult \circ (S \otimes 1) \left( (t^{-1} \otimes t^{-1} ) \Delta(t) \right) 
\end{equation}
is grouplike, and so can be used to develop a theory of quantum trace. 
\end{enumerate}
These conditions seem easier to check, and a Hopf algebra $\Hopf$ with such a $t$ already has some nice structure. This could be used, for instance, in studying coboundary categories (see \cite{D}).

\end{document}